\documentclass[oneside,reqno,11pt]{amsart}

\usepackage{cite}

\usepackage{amsthm}
\usepackage{amssymb}
\usepackage{bbm}
\usepackage{url}
\usepackage{graphicx,xcolor,marginnote, enumerate}
 \newtheorem{thrm}{Theorem}[]
 \newtheorem*{thrm*}{Theorem}
 \newtheorem{lemma}[thrm]{Lemma}
 \newtheorem*{lemma*}{Lemma}
 \newtheorem*{Vermutung*}{Vermutung}
 \newtheorem{cor}[thrm]{Corollary}
 \newtheorem{prop}[thrm]{Proposition}
 \theoremstyle{definition}
 
\newtheorem*{defi*}{Definition}

  \theoremstyle{remark}

  \newtheorem{remark}[thrm]{Remark}
  \newtheorem*{remark*} {Remark}

\newcommand{\dopp} [1] {\mathbbmss{#1}}

\newcommand{\f}{\begin{align}}
\newcommand{\fo}{\begin{align*}}
\newcommand{\fe}{\end{align}}
\newcommand{\foe}{\end{align*}}
\newcommand{\map}[3] {#1:\, #2\, \longrightarrow \, #3}

\newcommand{\R}{\mathbb{R}}

\newcommand{\E}{\mathbb{E}}
\newcommand{\N}{\mathbb{N}}
\newcommand{\C}{\mathbb{C}}

\newcommand{\U}{\mathcal{U}}

\renewcommand{\a}{\alpha}
\renewcommand{\b}{\beta}
\newcommand{\g}{\gamma}
\renewcommand{\d}{\delta}
\newcommand{\e}{\eta}

\newcommand{\s}{\sigma}
\renewcommand{\l}{\lambda}

\renewcommand{\hat}{\widehat}

\newcommand{\FNV}{\mathcal{F}_{N,V}}
\renewcommand{\phi}{\varphi}
\renewcommand{\e}{\epsilon}
\renewcommand{\epsilon}{\varepsilon}

\newcommand{\F}{\widehat}
\renewcommand{\U}{\mathcal{U}}

\newcommand{\lv}{\left\lvert}
\newcommand{\rv}{\right\rvert}

\newcommand{\Lip}[1]{\left\lvert #1 \right\rvert_{\mathcal{L}}}

\newcommand{\lb}{\left(}
\newcommand{\rb}{\right)}

\newcommand{\lvb}{\big\lvert}
\newcommand{\rvb}{\big\rvert}
\newcommand{\lbb}{\big(}
\newcommand{\rbb}{\big)}
\newcommand{\lee}{\big\{}
\newcommand{\ree}{\big\}}

\newcommand{\Ai}{\textup{Ai}}
\newcommand{\AI}{\mathbb{AI}}
\renewcommand{\O}{\mathcal{O}}
\renewcommand{\Re}{\textup{Re}}
\renewcommand{\Im}{\textup{Im}}
\renewcommand{\k}{\kappa}

\begin{document}
\bigskip

\title[Edge Statistics for Repulsive Particle Systems]{Edge Statistics for a Class of Repulsive Particle Systems}
   \author{Thomas Kriecherbauer}
   \address{Inst. for Mathematics, Univ. Bayreuth, 95440 Bayreuth, Germany}
   \email{thomas.kriecherbauer@uni-bayreuth.de}

  \author{Martin Venker}
   \address{Fac. of Mathematics, Univ. Bielefeld, P.O.Box 100131, 33501 Bielefeld, Germany}
   \email{mvenker@math.uni-bielefeld.de}

\thanks{The second author has been supported by the CRC 701.}


\begin{abstract}
 We study a class of interacting particle systems on $\R$ which was recently investigated by F. G\"otze and the second author
 \cite{GoetzeVenker}. These ensembles generalize eigenvalue ensembles of Hermitian random matrices by allowing different
 interactions between particles. Although these ensembles are not known to be determinantal one can use the stochastic 
linearization method of \cite{GoetzeVenker} to represent them as averages of determinantal ones. Our results describe the 
transition between universal behavior in the regime of the Tracy-Widom law and non-universal behavior for large 
deviations of the rightmost particle. Moreover, a detailed analysis of the transition that occurs in the regime of moderate 
deviations is provided. We also compare our results with the corresponding ones obtained recently for determinantal ensembles 
\cite{DissSchueler,EKS}. In particular, we discuss how the averaging effects the leading order behavior
in the regime of large deviations. In the analyis of the
averaging procedure we use detailed asymptotic information on the behavior of Christoffel-Darboux kernels that is uniform for
perturbative families of weights. Such results have been provided by K. Schubert, K. Sch\"uler and the authors in \cite{KSSV}.
\end{abstract}

 \maketitle

\section{Introduction and Main Results}
In \cite{BPS}, a class of particle ensembles on the real line with many-body interactions was introduced, generalizing invariant 
random matrix models, and results on the global asymptotics were obtained. G\"otze and the second author 
\cite{GoetzeVenker} presented a stochastic linearization procedure which allowed to study local correlations of particles in the 
special case of two-body interactions. They showed that in the limit of infinitely many particles the bulk correlations are 
described by the celebrated sine kernel law. The purpose of this paper is to extend the 
investigations of
\cite{GoetzeVenker} to the edge.

To be more precise, we consider particle ensembles on $\R$ with density proportional to
\begin{align}
\prod_{i<j}\phi(x_i-x_j)e^{- N\sum_{j=1}^NQ(x_j)},\label{e2}
\end{align} 
where $\map{Q}{\R}{\R}$  is a continuous function of sufficient growth at infinity compared to the continuous function
$\map{\phi}{\R}{[0,\infty)}$.
Apart from some technical conditions we will assume that
\begin{align}
\phi(0)=0,\quad \phi(t)>0\ \text{for }t\not=0\quad \text{and }\ \lim_{t\to0}\frac{\phi(t)}{\lv t\rv^2}=c>0,\label{e4}
\end{align}
or, in other words, $0$ is the only zero of $\phi$ and it is of order $2$. The connection to random matrix theory is as follows: If
$\phi(t)=t^2$, then we have the unitary invariant ensembles $P_{N,Q}$ given by

\begin{align}
 P_{N,Q}(x):=\frac{1}{Z_{N,Q}}\prod_{i<j}\lv x_i-x_j\rv^2 e^{- N\sum_{j=1}^NQ(x_j)}.\label{e1}
\end{align}
Here and in the following, we slightly abuse notation by using the same symbols for a measure and its density.
$P_{N,Q}$ is the joint distribution of the eigenvalues for the random $(N\times N)$ Hermitian matrix with density proportional to
\begin{align*}
 e^{-N\textup{Tr}(Q(X))}
\end{align*}
w.r.t. the Lebesgue measure on the space of $(N\times N)$ Hermitian matrices. Here $\textup{Tr}$ denotes the trace and $Q(X)$ is
defined by spectral calculus. This matrix distribution is invariant under unitary conjugations. See e.g. 
\cite{Deift,Handbook,PasturShcherbina,AGZ} for general references on random matrix theory. The unitary invariant ensembles
\eqref{e1} have been well-understood in terms of asymptotic behavior of the eigenvalues. A common way to describe their properties
is to use correlation functions. For a probability measure $P_N(x)dx$ on $\R^N$ and a natural number $k\leq N$ define
\begin{align*}
 \rho_N^k(x_1,\dots,x_k):=\int_{\R^{N-k}}P_N(x)dx_{k+1}\dots dx_N.
\end{align*}
The $k$-th correlation function is a probability density on $\R^k$, the corresponding measure is called $k$-th correlation measure.
Note that in the Random Matrix literature the $k$-th correlation function often differs from the definition above by the 
combinatorial factor $N!/(N-k)!$. The key for the detailed understanding of the local eigenvalue statistics of unitary invariant 
ensembles $P_{N, Q}$ is that they are determinantal ensembles with a kernel that can be expressed in terms of orthogonal 
polynomials. More precisely, denote by $K_{N,Q}$ the Christoffel-Darboux kernel of degree $N$ associated to the 
measure $e^{-NQ(x)}dx$, i.e.
\begin{align*}
 K_{N,Q}(t,s):=\sum_{j=0}^{N-1}p_j^{(N,Q)}(t)p_j^{(N,Q)}(s) e^{-\frac{N}{2}(Q(t)+Q(s))},
\end{align*}
where $p_j^{(N,Q)}$ is the unique polynomial of degree $j$ with positive leading coefficient such that $(p_j^{(N,Q)})_{j\in\mathbb 
N}$ forms an orthonormal sequence in $L^2(e^{-NQ(x)}dx)$.
Then, for all $1\leq k\leq N$ one can express the $k$-th correlation function by
\begin{align}
\frac{N!}{(N-k)!} \rho_{N,Q}^k(x_1,\dots,x_k) = \det\lb (K_{N,Q}(x_i,x_j))_{1\leq i,j\leq k}\rb. \label{determinantal_relations}
\end{align}
To state our results, it is convenient to rewrite \eqref{e2}. Let $\map{h}{\R}{\R}$ be continuous, even and bounded below. Let
$\map{Q}{\R}{\R}$ be continuous, even and of sufficient growth at infinity. Consider the probability density on $\R^N$ given by
\begin{align} 
  P_{N,Q}^h(x):=\frac{1}{Z_{N,Q}^h}\prod_{i<j}\lv x_i-x_j\rv^2\exp\{- N\sum_{j=1}^NQ(x_j)-\sum_{i<j}h(x_i-x_j)\} , \label{e3}
 \end{align}
where $Z_{N,Q}^h$ is the normalizing constant. Choosing $\phi(t):=t^2\exp\{-h(t)\}$, we see that $P_{N,Q}^h$ is in the form
$\eqref{e2}$. Conditions \eqref{e4} are satisfied as $h$ has no singularities.

G. Borot has pointed out to the second author that $P_{N,Q}^h$ is the eigenvalue distribution of
the unitary invariant ensemble of Hermitian matrices with density proportional to
\begin{align*}
  e^{-N\textup{Tr}(Q(X))-\textup{Tr}(h(X\otimes I-I\otimes X))},
\end{align*}
where $I$ denotes the identity matrix (cf. \cite{Borot1}). Nevertheless, we will prefer to view $P_{N,Q}^h$ as particle ensemble 
and 
therefore
speak of particles
instead of eigenvalues. Furthermore, the ensemble does not seem to be determinantal, in contrast to $P_{N,Q}$.

Our method of proof requires the external field $Q$ to be sufficiently convex. To quantify this, we define for $Q$ being twice
differentiable and convex $\a_Q:=\inf_{t\in\R}Q''(t)$. In the following we will denote by $\rho_{N,Q}^{h,k}$
the $k$-th correlation function or measure of $P_{N,Q}^h$.

The first information we need is the global behavior of the
particles, i.e. their asymptotic location and density. To this end, we state the following theorem which was proved in
\cite{GoetzeVenker}.

\begin{thrm}[{\cite[Theorem 1.1]{GoetzeVenker}}]\label{ethrm1}
 Let $h$ be a real analytic and even Schwartz function. Then there
exists a constant $\a^h\geq 0$ such that for all real analytic, strictly convex and even $Q$ with $\a_Q> \a^h$, the following
holds:\\
There exists a compactly supported probability measure $\mu_Q^h$ having a non-zero and continuous density on the interior of its
support and for
$k=1,2,\dots$, the $k$-th correlation measure of $P_{N,Q}^h$ converges weakly to the $k$-fold product of $\mu_Q^h$,
that is for any
bounded and continuous function $\map{g}{\R^k}{\R}$,
\begin{align}
\lim_{N\to \infty}\int g(t_1,\dots,t_k)\, \rho_{N,Q}^{h,k}(t_1,\dots,t_k)\, d^kt=\int g(t_1,\dots,t_k)\, d \mu_Q^h(t_1)\dots 
d\mu_Q^h(t_k).
\end{align}
\end{thrm}

\begin{remark}\label{remark_thrm1}\noindent
 \begin{itemize}
  \item[a)] If $h$ is positive-definite (i.e. with nonnegative Fourier transform), $\a^h$ can be chosen as 
$\a^h=\sup_{t\in\R}-h''(t)$ (cf. Remark \ref{remark_definiteness} and \cite[Remark after Theorem 1.1]{GoetzeVenker}).
\item[b)] It was shown in \cite[Sec. 3]{GoetzeVenker} that the measure $\mu_Q^h$ is the equilibrium measure with respect to 
the external field $V(x):=Q(x)+\int h(x-t) d\mu_Q^h(t)$. This fact leads to a useful representation of $\mu_Q^h$ (see e.g. 
Appendix \ref{Sec2}). We note for later reference that the evenness of $Q$ and $h$ carries over to $\mu_Q^h$ (cf. 
\cite[Theorem 1.1]{GoetzeVenker}) and hence to $V$.

\item[c)] In \cite{BPS}, ensembles with many-body interactions are considered and global asymptotics
 but not local correlations are derived (see also the discussion on page 326 in \cite{Totik}). In the case of pair 
interactions, the classes of admissible 
interactions in \cite{BPS}
and
in this paper are different. In \cite{BPS}, a convexity condition is posed, depending solely on the additional interaction
potentials where our conditions depend on both $Q$ and $h$. The characterization of the limiting measure is different, too.

Asymptotic expansions of the partition functions can be found in \cite{Borot2}. Further global asymptotics for 
repulsive particle systems have been studied in \cite{Ch1,Ch2}.
 \end{itemize}

\end{remark}

On the local scale, there are basically two different regimes of interest. The local regime in the bulk of the particles has 
been
investigated in \cite{GoetzeVenker}, yielding asymptotics in terms of the sine kernel. These results on the bulk have been extended 
in two ways. Firstly, in the recent \cite{SchubertVenker} the sine kernel asymptotics have been obtained under an unfolding of the 
particles, which is a non-linear rescaling that transforms the equilibrium measure to a uniform distribution. This in particular 
allows for an efficient treatment of empirical spacings between particles. Secondly, bulk universality has been shown in 
\cite{Venker12} for the $\b$-variant of \eqref{e3} which is obtained by replacing the exponent 2 in \eqref{e3} by some arbitrary 
$\b>0$.

Returning to $\b=2$, the
asymptotics at the edge are governed in terms of the Airy kernel. The Airy function $\map{\Ai}{\R}{\R}$ is uniquely determined by
the requirements that it is a solution to the differential equation $f''(t)=tf(t)$ and has the asymptotics $\Ai(t)\sim
(4\pi\sqrt{t})^{-1/2}e^{-\frac{2}{3}t^{3/2}}$ as $t\to\infty$. The Airy kernel $K_\Ai$ is defined on $\R^2$ and can be 
represented by
\begin{align}
 K_\Ai(t,s)=\int_0^\infty \Ai(t+r)\Ai(s+r)dr=\frac{\Ai(t)\Ai'(s)-\Ai'(t)\Ai(s)}{t-s},\label{def_KAi}
\end{align}
where the latter representation only holds for $t\not=s$.
The first main result of the present work is the edge universality of the correlation functions of $P_{N,Q}^h$. Here the 
universal aspect 
is that in the limit $N\to\infty$ the correlations do not depend on $Q$ and $h$ after an appropriate linear transformation 
$t\mapsto b+t/(c^* N^{2/3})$. The quantities $b$ and $c^*$ depend on $Q$ and $h$ through the function $V$ defined in Remark 
\ref{remark_thrm1} b) as follows. The number $b$ is the supremum of the support of the equilibrium measure $\mu_Q^h$ and $c^*$ is 
related to the prefactor of the square root vanishing of the density of $\mu_Q^h$ at $b$,
\begin{align}\label{def_c^*}
 c^*:=\dfrac{2^{-1/3}}{b}G_V(1)^{2/3},
\end{align}
where the function $G_V$ is defined in \eqref{functionG}.

\begin{thrm}\label{ethrm2}
 Let $h$ and $Q$ satisfy the assumptions of Theorem \ref{ethrm1} and assume furthermore that the Fourier transform $\hat{h}$ of $h$ decays
exponentially fast. Let
$q\in \R$ be given and let $b$ and $c^*$ be as above.
Then we have for $N\to\infty$
\begin{enumerate}
 \item \begin{align}
&\lb\frac{N^{1/3}}{c^*}\rb^k\rho_{N,Q}^{h,k}\lb
b+\frac{t_1}{c^*N^{2/3}},\dots,b+\frac{t_k}{c^*N^{2/3}}\rb=\det\left[K_\textup{Ai}(t_i,t_j)\right]_{1\leq i,j\leq 
k}+o(1)\label{Ttheorem3}                                                  
\end{align}
with the $o(1)$ term being uniform in $t\in[q,\infty)^k$. 
\item If $h$ is negative-definite, i.e. $\hat{h}\leq0$, then for any $0<\s<1/3$ the $o(1)$ term in \eqref{Ttheorem3} can be replaced 
by $\O(N^{-\s})$, again uniformly in $t\in[q,\infty)^k$.
\end{enumerate}
\end{thrm}

\begin{remark}\label{remark_definiteness_vorne}\noindent
 \begin{enumerate}
  \item  Even in the situation of part b) where we have better bounds on the rate of convergence, the 
Airy kernel determinant is the leading term in \eqref{Ttheorem3} only for $t\in[q,o((\log 
N)^{2/3})]^k$.
\item The condition of exponential decay of the Fourier transform of $h$ allows for writing $h$ as difference of positive-definite 
and real-analytic functions. This technical condition is needed and explained in Section \ref{proof_Theorem_2}. In particular, if 
$h$ is positive-definite or negative-definite, the assumption on the decay of the Fourier transform is automatically satisfied 
\cite{LukacsSzasz}.
\item If in the situation of part a) of the previous theorem, $h$ is positive-definite, then $\a^h$ can be chosen as 
$\a^h=\sup_{t\in\R}-h''(t)$ (cf. Remark \ref{remark_thrm1}). This remark extends to all of our results.
The roles of the different conditions on the definiteness of $h$ are explained in Remark \ref{remark_definiteness} below.
\item 
Edge universality has been established for a wide variety of random matrix ensembles. For $\b$-variants of \eqref{e1} with 
one-cut support we refer the reader to the recent works \cite{KRV}, \cite{Bourgadeetal13} and 
\cite{BFG}. It should be noted that the limit laws depend on the value of $\b$ \cite{RRV}. For results on invariant matrix ensembles and on 
Wigner matrices we refer to the books and recent surveys \cite{DeiftGioev, PasturShcherbina, TaoVuSurvey, ErdosSurvey} and 
references therein.
 \end{enumerate}
\end{remark}

We now turn to the largest particle $x_{\max}$ and provide an extensive description of its distributional properties. 
In analogy to the classical central limit theorem for independent variables, we prove a limit theorem on the fluctuations of 
$x_{\max}$ around the rightmost endpoint $b$ (Theorem \ref{ethrm3}). It also implies the weak law of large numbers with respect to 
a natural product measure. The strong law of large numbers is also shown (Corollary \ref{a.s.convergence}). Special 
emphasis is given to the study of the upper tail of the distribution function of $x_{\max}$. Strong results on moderate and 
large deviations are presented in Theorem \ref{thrm_MD} and 
Corollary \ref{corTW}. We already point out, that it is only in the regime of large deviations that the differences between the 
repulsive particle systems $P_{N,Q}^h$ and the determinantal ensembles become visible in the leading order behavior.

We start with the fluctuations of $x_{\max}$ around $b$. The limiting law is not Gaussian but is given by the Tracy-Widom 
distribution. Its 
distribution function $F_2$ can be defined as a Fredholm determinant implying in particular the representation
\begin{align}
 F_2(s)=1+\sum_{k=1}^\infty\frac{(-1)^k}{k!}\int_{(s,\infty)^k}\det\left(K_{\Ai}(t_i,t_j)\right)_{i,j\leq k}dt.\label{F_2_rep}
\end{align}
It can also be expressed as 
\begin{align*}
 F_2(s)=\exp\{-\int_s^\infty(t-s)q(t)^2dt\},
\end{align*}
where $q$ is the solution of the differential equation $q''(t)=tq(t)+2q(t)^3$ which has the asymptotics $q(t)\sim\Ai(t)$ as
$t\to\infty$ \cite{TracyWidom94}. For later use we note that the large-$s$-asymptotics of $F_2(s)$ are known (see e.g. \cite[(1), 
(25)]{BBDF}) to be
\begin{align}
1-F_2(s)=\frac{1}{16\pi}\frac{e^{-\frac43s^{3/2}}}{s^{3/2}}\left(1+\O\left(\frac{1}{s^{3/2}}\right)\right) ,
\quad \text{for} \ \ s \to \infty .\label{TW_asymptotics}
\end{align}

\begin{thrm}\label{ethrm3}

 Let $x_{\max}$ denote the largest component of a vector $x$. Under the conditions and with the notation of Theorem \ref{ethrm2} we
have for $N\to\infty$
\begin{enumerate}
 \item 
\begin{align}
P_{N,Q}^h\lb \lb x_{\max}-b\rb c^*N^{2/3}\leq s\rb=F_2(s)+o(1)\label{TW1}
\end{align}
uniformly for $s\in\R$, where $F_2$ is the distribution function of the ($\b=2$) Tracy-Widom distribution.
\item If $h$ is negative-definite, then the $o(1)$ term in \eqref{TW1} can be replaced by $\O(N^{-\s})$, where $0<\s<1/3$. 
The $\O$ term is uniform for $s\in[q,\infty)$ with arbitrary $q\in\R$.
\end{enumerate}
\end{thrm}

Note that Theorem \ref{ethrm3} implies the weak law of large numbers w.r.t.~the product measure 
$\mathbb{P}:=\bigotimes_{N=1}^\infty P_{N,Q}^h$ (on an appropriate measurable space), i.e. $x_{\max}$ converges in probability to 
$b$. The strong law, i.e. almost sure convergence, is the content of Corollary \ref{a.s.convergence} below.

Theorem \ref{ethrm3} can be viewed as an analogue of the central limit theorem for the distribution of the largest
particle. This point of view leads naturally to the question of what can be said about its large and moderate deviations? There one 
is interested, roughly speaking, in describing the decay of the probabilities $P_{N,Q}^h\lb \lb x_{\max}-b\rb c^*N^{2/3}\leq - s\rb$ 
(lower tail) and $P_{N,Q}^h\lb \lb x_{\max}-b\rb c^*N^{2/3} > s\rb$ (upper tail) when $s$ is allowed to grow to infinity at some 
$N$-dependent rate. In this paper we deal exclusively with the upper tail. Note that even for determinantal ensembles the lower 
tail has not been studied at the level of detail required for our analysis.

The large (and moderate) deviations theory generally deals with 
probabilities decaying at an exponential rate. In a general setting with probability measure $P$ and random variables $X_N$, the 
deviation $s$ is often parametrized as $s=rN^\a$ for some $r, \a>0$ and one proves statements like 
\begin{align}
 \lim_{N\to\infty}\frac{\log P(X_N>rN^\a)}{N^\g}=-I(r).\label{logarithmic_deviations}
\end{align}
Then $N^\g$ is called the speed and $I(r)$ the rate function. In our situation $P$ corresponds to $\mathbb{P}:=\bigotimes_{N=1}^\infty P_{N,Q}^h$ and $X_N$ corresponds to $\lb x_{\max}-b\rb c^*N^{2/3}$ with $x \in \R^N$. The regime of moderate deviations extends over $0 < \a < 2/3$ and the regime of large deviations is associated with $\a = 2/3$. Note that Theorem \ref{ethrm3} does not provide any information on either the speed or on the rate function for any $\a > 0$ due to the fast decay of $1 - F_2$ described in \eqref{TW_asymptotics}.

The next theorem and its two subsequent corollaries remedy this situation. E.g.
in the case of moderate deviations Corollary \ref{corLogarithm} b) shows for all $0 < \a < 2/3$ that we have speed $N^{3\a /2}$ and rate function $I(r)=\frac{4}{3} r^{3/2}$. For large deviations $(\a = 2/3)$ we learn from Corollary \ref{corLogarithm} a) that the speed is $N$ and the rate function $I(r)=\eta_V(1 + r/(b c^*))$ with $\eta_V$ defined in \eqref{functioneta}. Observe that in the moderate regime the rate function is universal, whereas the rate function depends on $V$ and is hence non-universal for large deviations.

The situation becomes more subtle, if one studies the asymptotics of the tail probabilities and not just their logarithms. 
These provide finer information about the tails and we show how the $V$-dependence of the leading order behavior becomes 
stronger within the region of moderate deviations (Remark \ref{remarkUniversality}).

In recent years superlarge deviations have been studied as well, which would correspond to $\a > 2/3$. We will not 
treat them here. See Remark \ref{remarkcomparisonSchueler} for a description of the difficulties that appear in that regime.

We now state Theorem \ref{thrm_MD} that summarizes our results on moderate and large deviations 
for the upper tail of the largest particle. It is stated for $P_{N,Q}^h(x_{\max} > t)$ and hence $s$ in the 
discussion above corresponds to $t = b + s/(c^* N^{2/3})$. In this scaling the regime of large deviations is described (more 
naturally as for $s$) by those values of $t$ for which $t-b$ is of order $1$ and the regime of moderate deviations reads $N^{-2/3 } 
<< t-b << 1$. 
Theorem \ref{thrm_MD} provides upper and lower bounds on $P_{N,Q}^h(x_{\max} > t)$ 
for $t$ belonging to an arbitrary but fixed bounded subset of $(b +  N^{-2/3}, \infty)$. 

In the regime of moderate deviations the quotient of the upper and lower bound is seen to converge to $1$ as 
$N \to \infty$ and thus the leading order behavior of $P_{N,Q}^h(x_{\max} > t)$ is determined. It is given by
\begin{align}\label{functionFNV}
\FNV(t) := \frac{b^2e^{-N\eta_V(t/b)}}{4\pi N(t^2-b^2)\eta_V'(t/b)}.
\end{align}
The $V$-dependence is contained in the function 
$\map{\eta_V}{[1,\infty)}{\R}$,
\begin{align}
\eta_V(t) := \int_1^t\sqrt{s^2-1}~G_V(s)ds	 ,\label{functioneta}
\end{align}
where the function $G_V$ is defined in \eqref{functionG} below for the external field $V$ introduced in Remark \ref{remark_thrm1} 
above.
In the regime of large deviations, our result is weaker. There we show that the quotient of $\FNV(t)$ and $P_{N,Q}^h(x_{\max} > t)$ 
and its inverse are both bounded.

\begin{thrm}\label{thrm_MD}
Assume that the conditions of Theorem 
\ref{ethrm2} are satisfied and
fix $T > b$. Then the following holds:
\begin{enumerate}
\item There exists a constant $B > 1$ such that for all $N$ and all $t \in (b + N^{-2/3}, T)$ we have
\begin{align}\label{deviations_lower}
\frac{1}{B} \FNV(t) \leq P_{N,Q}^h(x_{\max} > t) \leq B \FNV(t).
\end{align} 
 \item For any sequences $(p_N)_N$, $(q_N)_N$ of reals with $p_N < q_N$,  $N^{2/3}(p_N -b) \to \infty$, and $q_N -b \to 0$ for $N \to \infty$ we have 
\begin{align}\label{deviations_general}
 P_{N,Q}^h\lb  x_{\max} > t\rb=\FNV(t)  [ 1+ o(1) ]
\end{align}
uniformly in $t \in (p_N, q_N)$ as $N \to \infty$.
\item If $h$ is negative-definite, then $o(1)$ in statement b)  can be replaced by the more precise
\begin{align}\label{E6b}
\O\lb
\frac{1}{N\lv t-b\rv^{3/2}}\rb +\O\lb \sqrt{t-b} \rb
\end{align}
with the $\O$ terms being uniform in $N$ and in $t \in (b + N^{-2/3}, T)$. 
\end{enumerate}

\end{thrm}

Note that for negative-definite functions $h$ statement c) implies the upper bound in \eqref{deviations_lower} but not the 
lower bound.
\begin{remark}\label{remarkcomparisonSchueler}
 Theorem \ref{thrm_MD} should be compared with the result \cite[(1.14)]{DissSchueler} (cf. \cite{EKS}) for determinantal ensembles. 
Even our best 
error bounds in the case of negative-definite functions $h$ are worse than the bounds there, if $\lv t-b \rv >> N^{-1/2}$. This is 
due to the  $\O (\sqrt{t-b} )$ term  in \eqref{E6b}. We believe that this is not an artefact of our proof. In fact, relations 
\eqref{p6b5} and \eqref{p6b10} of the proof show that we can express the distribution of the largest particle by averaging 
determinantal ensembles with $V$  replaced by $V-f/N$ over $f$. Tracing back the origin of the $\O (\sqrt{t-b} )$ term  in 
\eqref{E6b} we see that it does not stem from a lack of control on error terms but on the dependence of the upper endpoint 
$b_{V-f/N}$ 
on $f$, that can actually be computed to first order.
In the regime of large deviations this effect is most prominent and this is the reason why we loose the convergence result for the 
distribution of the largest particle. However, we conjecture that also in the regime of large deviations we have convergence
\begin{align}\nonumber
\lim_{N\to\infty} P_{N,Q}^h\lb  x_{\max} > t\rb= B(t) \FNV(t)
\end{align}
for some strictly monotone increasing function $B$ with $B(t) > 1$. Observe that our previous results all show that 
$P_{N,Q}^h$ has the same limiting behavior as the determinantal ensemble $P_{N,V}$. However, for $P_{N,V}$ one has $B(t)=1$, i.e. 
we conjecture the breakdown of comparability between $P_{N,Q}^h$ and $P_{N,V}$ in this regime! To show this one would need to 
apply Laplace's method to the
infinite dimensional integral that represents the averaging over $f$. This is a technically demanding enterprise but not without hope 
due to the Gaussian character of the probability measure on $f$. We plan to come back to this problem in future work.

The situation becomes even more challenging in the superlarge regime, where the additional difficulty arises that one may not 
truncate $P_{N,Q}^h$ to some finite interval $[-L, L]^N$ that is used in the proof of Theorem \ref{thrm_MD} to obtain good bounds on 
the Gaussian field $f$ used for the linearization method (see e.g. discussion in the paragraph before \eqref{Gaussian_LD}). In 
contrast, for 
determinantal ensembles the leading order behavior of the upper tail $P_{N,V}\lb  x_{\max} > t\rb$ was shown to be given 
by $\FNV(t)$ in both the large and the superlarge regime for a suitable class of convex functions $V$ \cite[Theorem 
1.1]{DissSchueler} (cf. \cite{EKS}).

\end{remark}

Next we discuss the connection with the Tracy-Widom law $F_2$ beyond the central limit regime already covered in Theorem 
\ref{ethrm3}.
It is convenient to use the scaling of $x_{\max}$ as in the statement of Theorem \ref{ethrm3}, i.e. $t \equiv b + s / (c^* N^{2/3})$.
The positivity and real analyticity of $G_V$ (see \eqref{functionG} and thereafter), definitions \eqref{functioneta}, 
\eqref{functionFNV} and the definition of $c^*$ in \eqref{def_c^*} lead by a straightforward calculation to the following 
representation of $\FNV$ in the regime of moderate deviations, i.e $s/N^{2/3} = o(1)$:
\begin{align}\nonumber
\FNV(t(s)) &= \frac{e^{-N \eta_V (t(s)/b)}}{16 \pi s^{3/2}} \left[ 1 + \O \lb \frac{s}{N^{2/3}} 
\rb \right] \quad \text{with}\\
\label{expansion_eta}
N \eta_V (t(s)/b) &= \frac{4}{3}s^{3/2} + \sum_{j=1}^{\infty} d_{j, V} \frac{s^{j + \frac{3}{2}}}{N^{\frac{2}{3} j}} = 
 \frac{4}{3}s^{3/2} + \O \lb \frac{s^{5/2}}{N^{2/3}} \rb
\end{align}
for some sequence $(d_{j, V})_{j \geq 1}$ of real numbers depending on $V$. The authors would like to thank F.~G\"otze for 
pointing out that the power series
\begin{align*}
 \frac{4}{3}s^{3/2} + \sum_{j=1}^{\infty} d_{j, V} \frac{s^{j + \frac{3}{2}}}{N^{\frac{2}{3} j}}
\end{align*}
is an analogue to the Cram\'er series \cite{Cramer} in the deviations theory for sums of independent random 
variables. It follows that
\begin{align}\nonumber
e^{-N \eta_V (t(s)/b)}=e^{-\frac{4}{3}s^{3/2}}[1 + o(1)]
\end{align}
as $N \to \infty$ for $0 \leq s= o \lb N^{4/15} \rb$, or equivalently for $0 \leq t(s)-b= o \lb N^{-2/5} \rb$. Comparing with the 
asymptotics of the Tracy-Widom distribution \eqref{TW_asymptotics} and using in addition that 
\begin{align}\nonumber
\sqrt{\frac{s}{N^{2/3}}} = \O\left(\frac1{s^{3/2}}\right)+\O\left(\frac{s^{5/2}}{N^{2/3}}\right) \quad \text{for} \ \  s > 0 , 
\end{align}
we obtain from statements b) and c) of Theorem 
\ref{thrm_MD}:
\begin{cor}\label{corTW}
Under the condition and with the notation of Theorem 
\ref{thrm_MD} the following relations hold.
\begin{enumerate} 
 \item For any sequences $(\hat{p}_N)_N$, $(\hat{q}_N)_N$ of reals with $\hat{p}_N < \hat{q}_N$,  $\hat{p}_N \to \infty$, and 
$N^{-4/15} \hat{q}_N  \to 0$ for $N \to \infty$ we have: 
\begin{align}\nonumber
P_{N,Q}^h\lb  (x_{\max}-b)c^*N^{2/3} > s\rb=(1-F_2(s))(1+o(1)),
\end{align}
uniformly in $s \in (\hat{p}_N, \hat{q}_N)$ as $N \to \infty$.
\item If $h$ is negative-definite, then 
\begin{align}\nonumber
P_{N,Q}^h\lb  (x_{\max}-b)c^*N^{2/3} > s \rb=(1-F_2(s)) \left[1
+\O\left(\frac1{s^{3/2}}\right)+\O\left(\frac{s^{5/2}}{N^{2/3}}\right) \right]
\end{align}
with the $\O$ terms being uniform in $N$ and in $s \in (1 , N^{4/15})$.
\end{enumerate}
\end{cor}

\begin{remark}\label{remarkUniversality}
With Corollary \ref{corTW} we have identified the region for which the leading order of the upper tail of the distribution of the 
largest particle is universal, i.e. all of its $Q$ and $h$ dependence is encoded in the two numbers $b$ and $c^*$ that define the 
linear rescaling and play the same role as mean and variance in the Central Limit Theorem. Relations \eqref{expansion_eta} show 
nicely how the dependence of the leading order $\FNV(t)$ on $Q$ and $h$ grows gradually within the regime of moderate deviations. 
Set $\gamma_k := 2/(2k+5)$. For $0 < t -b = o \lb N^{-\g_k} \rb$, i.e. $s(t) = o\lb N^{-\g_k + 2/3} \rb$, the function $\FNV(t)$ 
depends to leading order on $Q$ and $h$ 
only through the $k+2$ numbers $b$, $c^*$, $d_{1,V}, \ldots, d_{k,V}$. This is a generic phenomenon already observed in \cite[Remark 
4.12]{DissSchueler} (cf. \cite{EKS}) for determinantal ensembles.
\end{remark}

In the logarithmic form, it
is straightforward to derive the following results from Theorem \ref{thrm_MD} and relations \eqref{expansion_eta}, 
\eqref{functionFNV}, \eqref{functioneta} and \eqref{functionG}. 
\begin{cor}\label{corLogarithm}
Under the condition and with the notation of Theorem 
\ref{thrm_MD} the following statements hold.
\begin{enumerate}
\item
\begin{align}\nonumber
\frac{\log P_{N,Q}^h(x_{\max} > t)}{N} = - \eta_V \lb t/b \rb - \frac{\log \lb N (t-b)^{3/2}\rb}{N} + \O \lb \frac{1}{N} \rb,
\end{align}
where the $\O$ term is uniform in  $N$ and in $t \in (b + N^{-2/3}, T)$.
\item
For any sequences $(\hat{p}_N)_N$, $(\hat{q}_N)_N$ of reals with $\hat{p}_N < \hat{q}_N$,  $\hat{p}_N \to \infty$, and $N^{-2/3} 
\hat{q}_N  \to 0$ for $N \to \infty$ we have
\begin{align}\nonumber
\frac{\log P_{N,Q}^h\lb  (x_{\max}-b)c^*N^{2/3} > s\rb}{s^{3/2}}= -\frac{4}{3} - \frac{\log (16 \pi s^{3/2})}{s^{3/2}} + o\lb \frac{1}{s^{3/2}} \rb + \O \lb  \frac{s}{N^{2/3}} \rb,
\end{align}
with error bounds that are uniform in $N$ and $s \in (\hat{p}_N, \hat{q}_N)$. If $h$ is negative-definite, then the $o \lb 
s^{-3/2} \rb$ term 
can 
be replaced by $\O \lb s^{-3} \rb $ that is also uniform in  $s \in (\hat{p}_N, \hat{q}_N)$.
\end{enumerate}
\end{cor}

Finally
we present the strong law of large numbers for the largest particle that will be proved at the end of Section \ref{TSec4}.

\begin{cor}\label{a.s.convergence}
Assume the conditions of Theorem \ref{ethrm2} and set $\mathbb{P}:=\bigotimes_{N=1}^\infty P_{N,Q}^h$. Then
\begin{align*}
 \mathbb{P}(\lim_{N\to\infty}x_{\max}=b)=1.
\end{align*}
\end{cor}

\begin{remark}
Corollary \ref{corLogarithm} a) implies a full large deviations principle on $\R$ for $x_{\max}$ with speed $N$ and good 
rate function 
$t\mapsto\eta_V(t/b)$, where we set $\eta_V(t)=\infty$ for $t<1$.
Such a large deviations principle has already been proved in a more general setting
for $\b$-variants of \eqref{e1} (see \cite{Borot2} and references therein). It can be derived via a contraction principle from a 
large deviations principle for the empirical distribution of the particles.
We also refer the reader to \cite{LedouxRider} and for results on combinatorial models that are related to invariant matrix 
ensembles to \cite{JohT, BDMMZ, Merkl1, Merkl2}.
As mentioned above, stronger results in non-logarithmic form as in Theorem \ref{thrm_MD} and Corollary \ref{corTW} have been shown 
for determinantal ensembles  in \cite{DissSchueler,EKS}.
\end{remark}

The paper is organized as follows.  Section \ref{Sec3} provides a sketch of the 
central ideas introduced in \cite{GoetzeVenker} that guide the proofs of all our main results. Then, in Section 
\ref{proof_Theorem_2}, we derive our universality result on the correlation functions. The last section deals with our various 
results on the distribution 
of the largest particle.
In Appendix \ref{Sec2}, we collect all results from 
\cite{KSSV} on determinantal ensembles that 
are needed in this paper and bring them in a form that is suitable for our purposes.

\section{Outline of the Method}\label{Sec3}

We first mention the basic steps from \cite{GoetzeVenker} in the analysis of $P_{N,Q}^h$. The additional interaction term
$\sum_{i<j}h(x_i-x_j)$ is in general of order $N^2$, so it may influence the limiting measure. The idea is to split it into a term
contributing to the formation of the limiting measure and a perturbation term of lower order. To this end, let us introduce for a
probability measure $\mu$ on $\R$ the notation $h_\mu(t):=\int h(t-s)d\mu(s)$ and $h_{\mu\mu}:=\int\int h(t-s)d\mu(t)d\mu(s)$. We
can then write
\begin{align}
 &\sum_{i<j}h(x_i-x_j)=\frac{1}{2}\sum_{i,j}h(x_i-x_j)-\frac{N}{2}h(0)\nonumber\\
=&N\sum_{j=1}^Nh_\mu(x_j)+\frac{1}{2}\sum_{i,j}\left[h(x_i-x_j)-h_\mu(x_i)-h_\mu(x_j)+h_{\mu\mu}\right]+C_N,\label{e5}
\end{align}
where $C_N:=-(N/2) h(0)-(N^2/2)h_{\mu\mu}$. The term in brackets in \eqref{e5} is the Hoeffding decomposition of the statistic
$\sum_{i<j}h(x_i-x_j)$ w.r.t. the measure $\mu$. The term $N\sum_{j=1}^Nh_\mu(x_j)$ will be added to the external field $Q$ forming
a new potential $V_\mu:=Q+h_\mu$. Setting
\begin{align}
 \U_\mu(x):=-\frac{1}{2}\sum_{i,j}\left[h(x_i-x_j)-h_\mu(x_i)-h_\mu(x_j)+h_{\mu\mu}\right],\label{def_U}
\end{align}
we arrive at the representation
\begin{align}\label{relation}
 P_{N,Q}^h(x)=\frac{Z_{N,V_\mu}}{Z_{N,V_\mu,\U_\mu}}P_{N,V_\mu}(x)e^{\U_\mu(x)}
\end{align}
with $Z_{N,V_\mu,\U_\mu}:=Z_{N,Q}^he^{C_N}$. Formula \eqref{relation} establishes a relation to a determinantal ensemble, 
indicating that asymptotics of $P_{N,Q}^h$ may be deduced from asymptotics of $P_{N,V_\mu}$. Our aim is to find a measure $\mu$ such 
that the ratio $Z_{N,V_\mu}/Z_{N,V_\mu,\U_\mu}$ and its reciprocal are bounded in $N$. It is straightforward to see that
\begin{align}
 \frac{Z_{N,V_\mu,\U_\mu}}{Z_{N,V_\mu}}=\E_{N,V_\mu}e^{\U_\mu},\label{M1}
\end{align}
where $\E_{N,V_\mu}$ denotes expectation w.r.t.~$P_{N,V_\mu}$. In view of the desired boundedness of \eqref{M1}, $\U_\mu$ should 
be centered at least asymptotically, which motivates the condition that $\mu$ should be the equilibrium measure to the field 
$V_\mu$.
This implicit problem was solved in 
\cite[Lemma 3.1]{GoetzeVenker} by
a fixed point argument, yielding existence but not uniqueness of such a $\mu$. The uniqueness followed later by proving 
that any
measure $\mu$ which is the equilibrium measure to $V_\mu$, is the limiting measure for $P_{N,Q}^h$. From now on let $\mu$ denote 
the unique measure with this property, write $V:=V_\mu$ and $\U:=\U_\mu$. Note that this definition of $V$ is consistent with that 
in Remark \ref{remark_thrm1} b).

At this stage, it can be proved that the ratio
$Z_{N,V}/Z_{N,V,\U}$ is bounded in $N$ and bounded away from $0$ provided that $\a_Q$ is large enough. More precisely, given
$\l>0$, there is an $\a(\l)<\infty$ such that there are constants $0<C_1(\l)<C_2(\l)<\infty$ such that for $\a_Q\geq \a(\l)$
\begin{align}\label{concentration_U}
C_1(\l)\leq\E_{N,V}e^{\l\U}\leq C_2(\l)
\end{align}
for all $N$ (see \cite[Proposition 4.7]{GoetzeVenker} and \cite[Remark 4.8]{GoetzeVenker}).

One main tool to derive bound \eqref{concentration_U} is the following representation of $\U$ in terms of linear 
statistics (cf. \cite[Lemma 4.6]{GoetzeVenker}) using Fourier techniques,
\begin{align}\label{Fourier}
 \U(x)=-\frac{1}{2\sqrt{2\pi}}\int \lv \sum_{j=1}^Ne^{i tx_j}-N\int e^{i 
ts}d\mu(s)\rv^2\F{h}(t)dt,
\end{align}
with $\F{h}(t)=(2\pi)^{-1/2}\int e^{-i ts}h(s)ds$.
The other ingredient to the proof of \eqref{concentration_U} is the 
following concentration of measure inequality for linear statistics which will be needed later on, too.

\begin{prop}\label{Concentration}
Let $Q$ be a real analytic external field with $Q''\geq c>0$. Then there exists a positive constant $C$ such that
for any Lipschitz function $g$ whose third derivative is bounded on an open interval $I\subset \R$ containing the support of the 
equilibrium measure
$\mu_Q$, we have for any
$\epsilon>0$ and $N$
\begin{align*}
\E_{N,Q}\exp\lee{\epsilon\lbb\sum_{j=1}^N g(x_j)-N\int g(t)d\mu_Q(t)\rbb}\ree\leq \exp\lee{\frac{\epsilon^2\Lip{g}^2}{2c}}+
\e C (\|g\|_\infty+\|g^{(3)}\|_\infty)\ree,
\end{align*}
where  $\Lip{g}$ denotes the Lipschitz constant of $g$ and $\|\cdot\|_\infty$ denotes the sup norm on $I$.
\end{prop}

The proposition can be found in \cite[Corollary 4.4]{GoetzeVenker} and follows from the fact that for strongly convex $Q$, 
$P_{N,Q}$ fulfills a log-Sobolev inequality which implies
concentration of the Lipschitz function $\sum_{j=1}^N g(x_j)$ around its expectation (see e.g. \cite[Proposition 4.4.26]{AGZ}). To
obtain Proposition
\ref{Concentration}, we combine this with the following estimate on the distance of the expectation to its large $N$ limit 
from
\cite[Theorem 1]{Shcherbina} (see also \cite[Theorem 1]{KriecherbauerShcherbina}),
\begin{align*}
 \Big\lvert \E_{N,Q}\sum_{j=1}^N g(x_j)-N\int gd\mu\Big\rvert\leq C(\|g\|_\infty+\|g^{(3)}\|_\infty),
\end{align*}
where $C$ is the constant that then appears in Proposition \ref{Concentration}.

The main idea to tackle the
local universality is linearizing
the bivariate statistics $\U$ to transform it into linear statistics. To this end, assume that $-h$ is positive-definite (i.e. 
$\F{h} \leq 0$; the
general case is reduced to this case). It is well-known (see e.g. \cite{AdlerTaylor}) that due to the positive-definiteness 
of $-h$, there exists a centered, stationary Gaussian process $(\tilde{f}(t))_{t\in\R}$ which has $-h$ as its covariance function. That is, 
the finite-dimensional distributions of $\tilde{f}$ are Gaussian, $\E \tilde{f}(t)=0$ and $\E \tilde{f}(t)\tilde{f}(s)=-h(t-s)$, where $\E$ denotes expectation 
w.r.t. the probability space underlying the Gaussian process.
The key observation now is that for fixed $x \in \R^N$ the random variable $\sum_{j=1}^N\tilde{f}(x_j)-N\int \tilde{f}d\mu$ 
is again a centered Gaussian with variance $2\,\U(x)$, yielding 
\begin{align}
 &\exp\{\U(x)\}=\E \exp\{\sum_{j=1}^N\tilde{f}(x_j)-N\int \tilde{f}d\mu\}=\E \exp\{\sum_{j=1}^Nf(x_j)\}\label{e7}\\
&\text{with}\quad f:=\tilde{f}-\int \tilde{f}d\mu.\label{def_f}
\end{align}
Combining \eqref{e7} with \eqref{relation} and \eqref{M1}, the ensemble $P_{N,Q}^h$ can now be represented as an average over determinantal ensembles $P_{N,V-f/N}$ as follows
\begin{align}
P_{N,Q}^h(x)=\frac{\E\left[Z_{N,V-f/N} P_{N,V-f/N}(x)\right]}{\E\, Z_{N,V-f/N}}.\label{Lago1}
\end{align}
Relation \eqref{Lago1} immediately extends to correlation functions. At this point the key observation for our universality proofs 
is the fact that for almost all $f$ the correlation functions $\rho_{N,V-f/N}^k$ and $\rho_{N,V}^k$ coincide in the region of 
interest to leading order if $N$ is sufficiently large. We will show in the appendix how to extract this information from 
\cite{KSSV} (cf. Proposition 
\ref{keyproposition}).

We end the brief overview of the method of proof by noting that in order to use the rich theory on Gaussian processes, in particular the sub-Gaussianity of maxima of Gaussian processes
on a compact,  we first need to truncate the ensemble to some compact interval $[-L,L]^N$.

Finally, we comment on the role that positive/negative-definiteness of $h$ plays in the approach of \cite{GoetzeVenker} and how it extends to the present paper.
\begin{remark}\label{remark_definiteness}
 The reader might have wondered about the different results depending on $h$ being positive-definite or negative-definite 
(cf. Remark \ref{remark_definiteness_vorne}).
If $h$ is positive-definite, then $\U$ in \eqref{def_U} is negative for all configurations $x$, which can be easily seen from the 
Fourier representation \eqref{Fourier}. This immediately gives the upper bound of \eqref{concentration_U}, without any condition 
on $\a_Q$. The lower bound and the further analysis in the case of positive-definite $h$ only need $V$ to be strongly convex, e.g. 
to apply Proposition \ref{Concentration}. With the crude bound $\a_Q>\sup_{t\in\R}-h''(t)$, given in Remark \ref{remark_thrm1}, 
this is guaranteed.\\
If $h$ is not negative-definite, then the linearization method as outlined in  \eqref{e7} needs to be modified. The 
function $h$ can be decomposed into positive-definite functions $h^\pm$, such that $h=h^+-h^-$. 
If $h^+\not=0$, then $-h$ is not a covariance function, but all functions $zh^++h^-$
with $z>0$ 
are and can be used for the linearization of a modified version $\U_z$ of $\U$. We will then employ Vitali's Theorem from 
complex analysis (a
consequence of Montel's Theorem and the Identity Principle) to transfer the  convergence results for $z>0$ to $z=-1$ that 
corresponds to the original $\U = \U_{-1}$. In this process, however, bounds on the rates of convergence are lost. This is the 
reason why we obtian in all of our results better bounds in the case of negative-definite functions $h$.
\end{remark}

\section{Detailed proof of Theorem \ref{ethrm2}}\label{proof_Theorem_2}
\subsection*{Representation of correlation functions and truncation}\noindent

We begin with a different representation of the correlation functions. For this we need a slightly generalized invariant
ensemble: Define for a continuous $Q$ of sufficient growth, continuous $f$ of moderate growth and $M\in\mathbb{N}$ the density on
$\R^N$ by
\begin{align}
 P_{N,Q,f}^M(x):=\frac{1}{Z_{N,Q,f}^M}\prod_{1\leq i<j\leq N}\lv x_i-x_j\rv^2e^{-M\sum_{j=1}^N Q(x_j)+\sum_{j=1}^N f(x_j)}.
\end{align}

We will usually have $M=N+k$ for some $k$. If $M=N$, we will abbreviate $P_{N,Q,f}:= P_{N,Q,f}^M$ and $ P_{N,Q}^M:= P_{N,Q,f}^M$ if
$f=0$. If $f=0$ and $M=N$, we have $P_{N,Q}$.
The $k$-th correlation function of $P_{N,V}$ at points $t_1,\dots,t_k$ can be written as 
\begin{align}
&\rho^k_{N,V}(t_1,\dots,t_k)\nonumber\\
 =&\int_{\R^{N-k}}\frac{1}{Z_{N,V}}\exp\lee-N\sum_{j=k+1}^NV(x_j)+2\sum\limits_{i<j;\ i,j>k} \log\lvb x_j-x_i\rvb
\ree\nonumber\\
&\times \exp\lee-N\sum_{j=1}^k
V(t_j)+2\sum_{i<j;\ i,j\leq
k} \log\lv t_i-t_j\rv\ree\nonumber\\
&\times\exp\lee 2\sum_{i\leq k,\ j>k}\log\lv t_i-x_j\rv\ree dx_{k+1}\dots dx_N\nonumber\\
&=F(t)\frac{Z_{N-k,V}^N}{Z_{N,V}}\E_{N-k,V}^N\exp\lee{2\sum_{i\leq k,\ j>k}\log\lv t_i-x_j\rv}\ree,\quad \text{where}\nonumber\\
&F(t):=\exp\lee{-N\sum_{j=1}^kV(t_j)+2\sum_{i<j;\ i,j\leq
k}\log\lv t_i-t_j\rv}\ree\label{e12}
\end{align}
 is the factor (\ref{e12}), which depends only on the fixed particles and $\E_{N-k,V}^N$ is the expectation of $P_{N-k,V}^N$.
We label the random eigenvalues of the ensemble $P_{N-k,V}^N$ by $x_{k+1},\dots,x_N$ and abbreviate 
$(t_1,\dots,t_k,x_{k+1},\dots,x_N)$ by $(t,x)$.
Setting
\begin{align}\label{def_R}
 R(t,x):=R_{N-k,V}^N(t,x):=2\sum_{i\leq k,\ j>k}\log\lv
t_i-x_j\rv+\log\big[F(t)\frac{Z_{N-k,V}^N}{Z_{N,V}}\big],
\end{align}
we arrive at the shorthand 
\begin{align}
&\rho^k_{N,V}(t_1,\dots,t_k)=\E_{N-k,V}^N\exp\lee{R(t,\cdot)}\ree.
\end{align}
Using \eqref{relation} and \eqref{M1} and recalling that we set $V=V_\mu$ and $\U=\U_\mu$, we see that the $k$-th correlation 
function $\rho_{N,Q}^{h,k}$ of $P_{N,Q}^h$ at $t_1\dots,t_k$ can be
written as 
\begin{align}
\rho_{N,Q}^{h,k}(t_1,\dots,t_k)=\frac{1}{\E_{N,V}\exp\lee{\U}\ree}\E_{N-k,V}^N\exp\lee{(\U+R)(t,\cdot)}\ree.\label{e8}
\end{align}
The representation \eqref{e8} allows to control the effects of truncation on the correlation functions (see \cite[Lemma 
6.3]{GoetzeVenker}).  More precisely, let $\E_{N,V;L}^M$ denote the expectation w.r.t. the ensemble $P_{N,V;L}^M$ obtained by 
normalizing the ensemble $P_{N,V}^M$
restricted to $[-L,L]^N$. Furthermore, let $R_L$ be the analogue of $R$ in which the partition functions in \eqref{def_R} are 
replaced by their truncated versions. Then \cite[Lemma 6.3]{GoetzeVenker} (see also \cite{Johansson98},\cite{BPS}) states that for 
each $k$ there are $L,C>0$
such that
for all $N$ and
for all $t_1,\dots,t_k$
\begin{align}
 \lv\rho_{N,Q}^{h,k}(t_1,\dots,t_k)-\frac{1}{\E_{N,V;L}\exp\lee{\U}\ree}\E_{N-k,V;L}^N\exp\lee{(\U+R_L)(t,\cdot)}
\ree\rv\leq
e^{-CN}.\label{truncation}
\end{align}
Moreover, we learn from the last inequality in the proof of \cite[Lemma 6.3]{GoetzeVenker} that $L$ can be chosen such that in addition there exists a constant $c>0$ with
\begin{align}
\lvert\rho_{N,Q}^{h,k}(t)\rvert\leq e^{-cN}\label{Lago2}
\end{align}
 for all $t\in\R^k\setminus(-\infty,L]^k$. It can be shown (see \cite[Remark 4.5, Remark 4.8]{GoetzeVenker} and the proof of 
\cite[Proposition 4.7]{GoetzeVenker}) that for such a choice of $L$ and any $\l>0$ we can again find constants $C_1(\l),C_2(\l)>0$ 
with
\begin{align}
0<C_1(\l)\leq\E_{N,V;L}e^{\l\U}\leq C_2(\l)\label{Lago3}
\end{align}
for all $N$, provided $\a_Q\geq\a(\l)$, where $\a(\l)$ can be chosen as in \eqref{concentration_U} and does therefore not depend on $L$.
\subsection*{Linearization and proof of Theorem \ref{ethrm2} b)}\noindent

We now give a more detailed description of the linearization procedure for negative-definite $h$. In this case 
we can
indeed view $-h$
as covariance function of a centered stationary Gaussian process on $\R$ such that  \eqref{e7} holds. Analyticity
of the sample paths can be deduced from the spectral representation of $\tilde{f}$ which we now explain in some detail. 

Let $(B_t)_{t\geq0}$ denote a standard 1-dimensional Brownian motion, in particular $B_0=0$ a.s.. Recall that by the law of the
iterated logarithm, we know that  $\lv B_t\rv$ is almost surely $\O(\sqrt{2t\log\log t})$. Thus, for 
$\map{g}{[0,\infty)}{\R}$
of sufficient decay and smoothness, we can define the Wiener integral
\begin{align}
 \int_0^\infty g(s)dB_s:=-\int_0^\infty B_sdg(s)=-\int_0^\infty B_sg'(s)ds.\label{gaussian_process}
\end{align}
Recall that $(B_t)_t$ is a Gaussian process, that is, its finite dimensional distributions are Gaussian which is equivalent to the
property that finite linear combinations of the family of random variables $\{B_t:t\geq0\}$ are Gaussian. By a limit argument we
see that \eqref{gaussian_process} is a Gaussian random variable and by elementary computations we find that the mean is 0 and the
variance is $\int_0^\infty g(s)^2ds$.

Now, for the representation of $\tilde{f}$, let $(B^1_t)_t, (B^2_t)_t$ denote two independent Brownian motions. Then we can define
\begin{align}
 \tilde{f}(t):=\lb\frac2\pi\rb^{1/4}\int_0^\infty \cos(ts)\sqrt{-{\hat{h}(s)}}dB_s^1+\lb\frac2\pi\rb^{1/4}\int_0^\infty
\sin(ts)\sqrt{-{\hat{h}(s)}}dB_s^2.\label{process_representation}
\end{align}
To verify that $\tilde{f}$ defined in this way has the desired properties, it is enough to note that the right hand side of 
\eqref{process_representation} forms a Gaussian
process on $\R$ (which can be easily checked using the characterization mentioned above) with mean $0$ and covariance function
$-h$. It is a somewhat surprising and maybe not so well-known fact \cite{LukacsSzasz} that the Fourier transform of a 
positive-definite real-analytic function decays exponentially at infinity (see also the discussion preceding \eqref{e472}).  By 
this  exponential decay of $\hat{h}$, representation \eqref{process_representation} continues to 
hold for $w$ from a strip $\{x+iy:x\in\R, \lv y\rv<\d\}$ for some $\d>0$. For later use, we define the region 
$D:=(-L-1,L+1)\times(-\d/3,\d/3)\subset \C$ with  an appropriate choice of $L$ made precise below. We thus see that 
$\tilde{f}$ is analytic on $D$ a.s.. It follows 
also
from \eqref{process_representation} that the extended process $(\tilde{f}(w))_{w\in D}$ is a complex-valued centered Gaussian process and
it
is straightforward to show that the covariance function is 
\begin{align}
\E (\tilde{f}(w_1)\overline{\tilde{f}(w_2)})=-h(w_1-\overline{w_2}).\label{covariance}
\end{align}

Before starting the proof we finally introduce an abbreviation for the Airy kernel determinants (see also \eqref{def_KAi})
\begin{align}
\AI_k(t):=\det\left[K_\textup{Ai}(t_i,t_j)\right]_{1\leq i,j\leq k}.\label{AI}
\end{align}

\begin{proof}[Proof of Theorem \ref{ethrm2} b)]
Let $k$ be fixed and choose $L$ such that \eqref{truncation} and \eqref{Lago2} are satisfied. We first restrict our attention to the 
case $t\in [q,N^\e]$ with $0<\e<\min(2/3-2\s,2/15)$. Note that this includes the region where the Airy kernel determinant 
describes the leading order behavior.
Using	\eqref{truncation} and the lower bound of $\eqref{Lago3}$, it suffices to show
\begin{align}
&(c^*)^{-k}N^{
k/3 }\E_{N-k,V;L}^N\exp\lee{(\U+R_L)(b+\frac{t}{c^*N^{2/3}},\cdot)}\ree-\E_{N,V;L}\exp\lee{\U}\ree\AI_k(t)\label{e9}\\
&=\O\lb
N^{-\s}\rb\nonumber
\end{align}
uniform in $t$.
Let $\tilde{f}$ be defined as in \eqref{process_representation} and set $f:=\tilde{f}-\int 
\tilde{f}d\mu$ (cf. \eqref{def_f}).
We now apply the linearization procedure and obtain from \eqref{Lago1}
\begin{align*}
\rho_{N,Q;L}^{h,k}(t)=\frac{\E\left[Z_{N,V-f/N;L} \rho_{N,V-f/N;L}^k(t)\right]}{\E\, Z_{N,V-f/N;L}}.
\end{align*}
Observe that
\begin{align}
Z_{N,V-f/N;L}=Z_{N,V;L}\,\E_{N,V;L}\exp\lbb\sum_{j=1}^Nf(x_j)\rbb\label{Lago4}
\end{align}
together with \eqref{e7} leads to
\begin{align*}
\E\,Z_{N,V-f/N;L}=Z_{N,V;L}\,\E_{N,V;L}\exp(\U).
\end{align*}
The last three relations together with a truncated version of \eqref{e8} yield
\begin{align*}
 &\E_{{N}-k,V;L}^{N}\exp\lee{(\U+R_L)(b+\frac{t}{c^*{N}^{2/3}},\cdot)}\ree=\E_{N,V;L}\exp(\U)\rho_{N,Q;L}^{h,k}(b+\frac{t}{c^*{N}^{2/3}})\nonumber\\
=&\E\Big[\E_{N,V;L}
\exp\lee\sum_{j=1}^{N}f(x_j)\ree \rho_{N,V,f;L}^k\lbb
b+\frac{t_1}{c^*N^{2/3}},\dots,b+\frac{t_k}{c^*N^{2/3}}\rbb\Big],
\end{align*}
where $\rho_{N,V,f;L}^k$ is the $k$-th correlation function of the unitary
invariant ensemble $P_{N,V-f/N;L}=:P_{N,V,f;L}$ defined
on $[-L,L]^N$. We thus get that \eqref{e9} is equal to
\begin{align}
\E\Big[\E_{N,V;L}
\exp\lee\sum_{j=1}^{N}f({x_j})\ree\Big( (c^*)^{-k}N^{k/3
}\rho_{N,V,f;L}^k\lbb
b+\frac{t_1}{c^*N^{2/3}},\dots,b+\frac{t_k}{c^*N^{2/3}}\rbb- \AI_k(t)\Big)\Big].\label{e11}
\end{align}
By Proposition \ref{keyproposition} (observe that $c^*=\g_V$ and $a_V=-b_V$), by the determinantal relations \eqref{determinantal_relations},
and by the almost sure analyticity of $f$, we obtain for the term in the inner
parenthesis 
\begin{align*}
(c^*)^{-k}N^{k/3
}\rho_{N,V,f;L}^k\lbb
b+\frac{t_1}{c^*N^{2/3}},\dots,b+\frac{t_k}{c^*N^{2/3}}\rbb- \AI_k(t) = \O(N^{-\s})
\end{align*}
for almost all $f$ with the $\O$-term uniform for $\|f\|_D\leq N^\kappa$ with $\k:=1/3-\s-\e/2$.
Since $\E\,\E_{N,V;L}
\exp\lee\sum_{j=1}^{N}f({x_j})\ree=\E_{N,V;L}
\exp\lee\U\ree$ is uniformly bounded in $N$ (cf. \eqref{Lago3}), we have derived
\begin{align}
\E\Big[\,&\dopp{1}_{\{\|f\|_D\leq N^\kappa\}}\E_{N,V;L}
\exp\lee\sum_{j=1}^{N}f({x_j})\ree\nonumber\\
&\times\Big( (c^*)^{-k}N^{k/3
}\rho_{N,V,f;L}^k\lbb
b+\frac{t_1}{c^*N^{2/3}},\dots,b+\frac{t_k}{c^*N^{2/3}}\rbb- \AI_k(t)\Big)\Big]=\O(N^{-\s}).\label{Lago42}
\end{align}
Hence the proof is finished for the case $t\in[q,N^\e]^k$ if we can show
\begin{align}
\E\Big[\,&\dopp{1}_{\{\|f\|_D>N^\kappa\}}\times\E_{N,V;L}
\exp\lee\sum_{j=1}^{N}f({x_j})\ree\nonumber\\
&\times\Big( (c^*)^{-k}N^{k/3
}\rho_{N,V,f;L}^k\lbb
b+\frac{t_1}{c^*N^{2/3}},\dots,b+\frac{t_k}{c^*N^{2/3}}\rbb- \AI_k(t)\Big)\Big]=\O(e^{-cN^{2\k}})\label{Gaussian_truncation}
\end{align}
for some $c>0$ uniformly in $t\in [q,N^{\e}]^k$. This bound will follow from an application of H\"older's inequality to 
separate
the three $f$-dependent factors in \eqref{Gaussian_truncation}. Here we will use $L^3$ norms for convenience.

We first treat $\dopp{1}_{\{\|f\|_D>N^\kappa\}}$. It follows
readily from \eqref{process_representation} that real and imaginary parts of $\tilde{f}$ on $D$ are (real-valued) centered Gaussian 
processes.
By Borell's inequality (see e.g. \cite[Theorem 2.1.1]{AdlerTaylor}) the supremum $\|X\|_\infty$
of a real-valued continuous centered Gaussian process $X_t$ over a compact $K$ is sub-Gaussian, i.e. dominated by a Gaussian random variable
with a certain expectation and variance $\s:=\sup_{t\in K}\E
X_t^2$. To apply this fact, we will without further notice always bound the supremum over the open set $D$ by the supremum 
over $\overline{D}$. As the sum of sub-Gaussian random variables is also sub-Gaussian, we see using $\sup_{w\in D}\lv 
\tilde{f}(w)\rv\leq
\sup_{w\in D}\lv \Re \tilde{f}(w)\rv+\sup_{w\in D}\lv \Im \tilde{f}(w)\rv$ that  $\sup_{w\in D}\lv \tilde{f}(w)\rv$ has 
sub-Gaussian tails. Since $\lv\int \tilde{f}d\mu\rv\leq\sup_{w\in D}\lv 
\tilde{f}(w)\rv$  we conclude that $\sup_{w\in D}\lv f(w)\rv$ is sub-Gaussian. Note that $\lvert\Re\tilde{f}\rvert$ and 
$\lvert\Im\tilde{f}\rvert$ are both bounded by $\lvert \tilde{f}\rvert$ and that  $\E \lvert \tilde{f}(w)\rvert^2=-h(2i\Im w)$ by 
\eqref{covariance}. We therefore find that $\sup_{w\in D} \lvert f(w)\rvert$ is dominated by a Gaussian with a variance $4\sup_{w\in 
[-\d/3,\d/3]}-h(2iw)$. Thus there exists $c>0$ such that
\begin{align}\label{Gaussian_LD}
 P\{\|f\|_D>N^\k\}=\O(e^{-3cN^{2\k}})\ \text{ and  hence }\ \lbb\E\dopp{1}_{\{\|f\|_D>N^\kappa\}}^3\rbb^{1/3}=\O(e^{-cN^{2\k}}).
\end{align}

We now turn to the second factor  $\E_{N,V;L}
\exp\lee\sum_{j=1}^{N}f({x_j})\ree$ in \eqref{Gaussian_truncation}.
Proposition \ref{Concentration} extends to the case of a truncated ensemble,
yielding an error of exponential order which we omit in the following (cf. \cite[Remark 4.5]{GoetzeVenker}). Thus we have
\begin{align}
\E\,\Big[ \E_{N,V;L}
\exp\lee\sum_{j=1}^{N}f({x_j})\ree\Big]^3\leq \E\exp\lee{\frac{3\Lip{f}^2}{2\a_V}}+
 3C(\|f\|_{[-L,L]}+\|f^{(3)}\|_{[-L,L]})\ree.\label{e10}
\end{align}
Observe  that the
Lipschitz constant is evaluated over $[-L,L]$ instead of $\R$. 

In order to estimate this $L^3$ norm, we argue in addition to the sub-Gaussianity of $\|\tilde{f}\|_{[-L,L]}$  that also the 
processes $\|\tilde{f}'\|_{[-L,L]}, \|\tilde{f}^{(3)}\|_{[-L,L]}$ are sub-Gaussian. To this end note that the process $\tilde{f}'$ 
is a centered stationary Gaussian process with covariance function $h''$. 
Borell's inequality yields again sub-Gaussianity of $\Lip{f}$  and analogous arguments prove
sub-Gaussianity of $\|f^{(3)}\|_{[-L,L]}$ as well. This shows the finiteness of the r.h.s.~of \eqref{e10} provided that $\a_Q$ and 
hence $\a_V$ is sufficiently large.
It is noteworthy that the condition on $\a_Q$ is determined by the variance $h''(0)$ of $\tilde{f}'$ only. As $\tilde{f}'$ is 
stationary, this $\a_Q$ is independent of $L$ and hence of $k$. In fact, this is the very reason that in  Theorem \ref{ethrm2} the 
convergence of \textit{all} correlation functions can be derived for a given function $h$.

Next, we will estimate $(c^*)^{-k}N^{k/3
}\rho_{N,V,f;L}^k\lbb
b+\frac{t_1}{c^*N^{2/3}},\dots,b+\frac{t_k}{c^*N^{2/3}}\rbb$.
From \eqref{determinantal_relations} we see that 
$(\tilde{K}_{N}(t_i,t_j))_{1\leq i,j\leq k}=:A$ (where $\tilde{K}_N$ is now a shorthand for the kernel $\tilde{K}_{N,V,f;L}$ associated to $P_{N,V,f;L}$ 
evaluated at rescaled variables $b+\frac{t_j}{c^*N^{2/3}}$, $1\leq j\leq k$) is positive
semi-definite
and can hence be
written as $A=B^2$ for some Hermitian matrix $B$. Now using Hadamard's inequality we obtain
\begin{align}\label{THad}
 &\det A=\lb \det B\rb^2\leq \prod_{j=1}^k\sum_{i=1}^k\lv B_{ij}\rv ^2=\prod_{j=1}^kA_{jj}.
\end{align}
In our case this reads
\begin{align}
 &\rho_{N,V,f;L}^k\lbb
b+\frac{t_1}{c^*N^{2/3}},\dots,b+\frac{t_k}{c^*N^{2/3}}\rbb\leq (N-k)!/(N!) \prod_{j=1}^k\tilde{K}_N(t_j,t_j)\nonumber\\
&\leq e^k\prod_{j=1}^k
\rho_{N,V,f;L}^{1}(b+\frac{t_j}{c^*N^{2/3}}).\label{Hadamard}
\end{align}
Now, we use the well-known (see e.g. \cite{Freud}) representation
\begin{align*}
	\rho^1_{N,V,f;L}(t)= \frac{e^{-NV+f}}{N\l_N(e^{-NV+f},t)}, \l_N(e^{-NV+f},t):=\inf_{P_{N-1}(t)=1}\int_{-L}^{L} \lv
P_{N-1}(s)\rv^2e^{-NV(s)+f(s)}ds,
\end{align*}
where the infimum is taken over all polynomials $P_{N-1}$ of at most degree $N-1$ with the property that $P_{N-1}(t)=1$. From this
it is obvious that
\begin{align*}
 \rho_{N,V,f;L}^{1}(b+\frac{t_j}{c^*N^{2/3}})\leq \rho_{N,V;L}^{1}(b+\frac{t_j}{c^*N^{2/3}})e^{2\|f\|_{[-L,L]}}.
\end{align*}
By Proposition \ref{keyproposition}, 
\eqref{determinantal_relations}, and by the uniform boundedness of the Airy kernel in the region of interest, we find that
\begin{align}
&(c^*)^{-1}N^{1/3
}\rho_{N,V,f;L}^{1}(b+\frac{t_j}{c^*N^{2/3}})\leq C e^{2\|f\|_{[-L,L]}}\ \text{ and thus}\nonumber\\
&(c^*)^{-k}N^{k/3
}\rho_{N,V,f;L}^k\lbb
b+\frac{t_1}{c^*N^{2/3}},\dots,b+\frac{t_k}{c^*N^{2/3}}\rbb\leq (Ce^{2\|f\|_{[-L,L]} +1})^k,\label{e15}
\end{align}
where $C$ does not depend on $f$.

This estimate together with the uniform boundedness of $\AI_k$ and the sub-Gaussianity proves the uniform boundedness of the third factor 
\begin{align*}
(c^*)^{-k}N^{k/3
}\rho_{N,V,f;L}^k\lbb
b+\frac{t_1}{c^*N^{2/3}},\dots,b+\frac{t_k}{c^*N^{2/3}}\rbb- \AI_k(t)
\end{align*}
in \eqref{Gaussian_truncation} in the $L^3$ norm. This completes the case $t\in[q,N^\e]^k$.

Assume now that at least one component of $t$, say $t_j$, is larger than $N^\e$. Then the asymptotics of the Airy kernel (see e.g. 
\eqref{KAi_asymp}) imply that all entries of the $j$-th column of $(K_{\textup{Ai}}(t_i,t_l))_{1\leq i,l\leq k}$ and consequently 
$\AI_k(t)$ are bounded by $\O(\exp(-N^{\e}))$ (cf. Remark \ref{remark_definiteness_vorne} a)). We are left to prove a similar bound 
for the l.h.s.~of \eqref{Ttheorem3}. If one of the components $t_j$ is such that $b+\frac{t_j}{c^*N^{2/3}}>L$, this follows from 
\eqref{Lago2} and for the remaining cases the estimate readily follows from Proposition \ref{prop_LD} a), \eqref{Hadamard} 
and Proposition \ref{keyproposition} for those components of $t$ which lie in $[q,N^\e]$.
Of course, Proposition \ref{keyproposition} can only be applied for $\| f \| \leq N^{\k}$. One may use the same arguments as in the 
case $t\in[q,N^\e]^k$ to bound the contribution of $\{ f : \| f \| > N^{\k}\}$.
\end{proof}

\subsection*{Extension to general $h$}\noindent

If $-h$ is not positive-definite, we may write it as a difference of positive-definite functions. Denoting by $g_{\pm}$
nonnegative and negative
part of a function, we first write $\F{h}=(\F{h})_+-(\F{h})_-$. Setting
$h^\pm$ as the inverse Fourier transform of $\F{h}_\pm$, we get a 
decomposition
$h=h^+-h^-$ of $h$ into positive-definite, real-analytic functions. It is this step where the assumption of exponential decay
of $\hat{h}$ is needed. Exponential decay of $\hat{h}$ is equivalent to $h^\pm$ being real-analytic. Sufficiency is easily seen as
the exponential decay allows for an entension of the Fourier representation of $h$ to a strip $D$ from which analyticity of $h^\pm$
can be deduced. For the necessity we remark that any real-analytic positive-definite function has an exponentially decaying Fourier
transform \cite[Theorem 2]{LukacsSzasz} and thus with $h^\pm$ also $h$ must have this property.

 Define for a complex parameter $z\in\C$
\begin{align}
 \U_z(x)=&\frac{z}{2}\lbb\sum_{i,j=1}^Nh^+(x_i-x_j)-\left[h^+_\mu(x_i)+h^+_\mu(x_j)-h^+_{\mu\mu}\right]\rbb\label{e472}\\
+&\frac{1}{2}\lbb\sum_{i,j=1}^Nh^-(x_i-x_j)-\left[h^-_\mu(x_i)+h^-_\mu(x_j)-h^-_{\mu\mu}\right]\rbb\label{e482}.
  \end{align}
We have $\U_{-1}=\U$.
As seen in \eqref{e9}, we have to show
\begin{align}
&(c^*)^{-k}N^{
k/3 
}\E_{N-k,V;L}^N\exp\lee{(\U_{z}+R_L)(b+\frac{t}{c^*N^{2/3}},\cdot)}\ree-\E_{N,V;L}\exp\lee{\U_{z}}\ree\AI_k(t)=o(1)\label{e92}
\end{align}
for $z=-1$ uniformly in $t\in[q,\infty)^k$. Let us recall two basic facts 
from analysis. As our
linearization procedure only allows for nonnegative real $z$, we will prove
\eqref{e92} for positive real $z$ and use complex analysis to deduce \eqref{e92} also for $z=-1$.  Recall Vitali's Theorem for instance from
\cite[5.21]{Titchmarsh}: Let $(f_n)_n$ be a sequence of analytic functions on a domain $U\subset\C$ with $\lv f_n(z)\rv\leq M$ for
all
$n$ and all $z\in U$.
Assume that $\lim_{n\to\infty}f_n(z)$ exists for a set of $z$ having a limit point in $U$. Then $\lim_{n\to\infty}f_n(z)$ exists
for all $z\in U$ and the limit is an analytic function in $z$.

To capture the uniformity of the convergence in \eqref{e92} in a way which preserves analyticity in $z$, we use the following
obvious characterization. A sequence of complex-valued functions $(f_n)_n$ converges uniformly on the
sequence of sets $(A_n)_n$, $A_n\subset \R^l$ towards a function $f$ if and only if for all sequences
$(n_m)_m\subset \mathbb{N}$ with $\lim_{m\to\infty}n_m=\infty$ and all sequences $(t_m)_m$ with $t_m\in A_{n_m}$ we have
$\lim_{m\to\infty}f_{n_m}(t_m)-f(t_m)=0$.

For our application, we take $A:=A_N:=[q,\infty)^k$. Let $(N_m)_m\subset\N$ be a
sequence going to
infinity and $(t_{(m)})_m$ be a sequence with $t_{(m)}\in A$.
Define $\map{W_m}{\C}{\C}$ by
\begin{align}
 W_m(z):=&(c^*)^{-k}N_m^{k/3
}\E_{{N_m}-k,V;L}^{N_m}\exp\lee{(\U_{z}+R_L)(b+\frac{t_{(m)}}{c^*{N_m}^{2/3}},\cdot)}\ree\nonumber\\
&-\E_{N_m,V;L}\exp\lee{\U_{z}}
\ree\AI_k(t_{(m)}).\label{W_m}
\end{align}
With these two observations, we are ready to complete the proof of Theorem \ref{ethrm2}.
\begin{proof}[Proof of Theorem \ref{ethrm2} a)]
We start by showing convergence of $W_m(z)$ uniformly for $z\in[0,1]$.
We will mostly omit the index $m$ in the following. Define $-h_z:=zh^++h^-$ such that $h_{-1}=h$, denote by $\tilde{f}_z$ 
the corresponding Gaussian process, and set $f_z:=\tilde{f}_z-\int \tilde{f}_zd\mu$. Now, using the same arguments given in the 
proof of Theorem \ref{ethrm2} b) we find that 
\begin{align*}
&\E\Big[\E_{N,V;L}
\exp\lee\sum_{j=1}^{N}f_z({x_j})\ree\Big( (c^*)^{-k}N^{k/3
}\rho_{N,V,f_z;L}^k\lbb
b+\frac{t_1}{c^*N^{2/3}},\dots,b+\frac{t_k}{c^*N^{2/3}}\rbb- \AI_k(t)\Big)\Big]\\
&=\O(N^{-\s}).
\end{align*}
The convergence is uniform in $z\in[0,1]$. This is a consequence of the boundedness of $\sup_{w\in [-\d/3,\d/3]}-h_z(2iw)$ for 
$z\in[0,1]$ (cf. the arguments above \eqref{Gaussian_LD}). Note that the choice of $\a_Q$ in the hypothesis of Theorem \ref{ethrm2} 
depends on this bound.
This settles the convergence of $W_m$ on $[0,1]$.

We complete the proof by demonstrating uniform boundedness of $W_m$ on the domain $G:=\{z\in\C:\Re z< 1\}$. We first observe that applying \eqref{def_U} and \eqref{Fourier} to $h^+$ instead of $h$ yields the non-negativity of $\sum_{i,j=1}^Nh^+(x_i-x_j)-\left[h^+_\mu(x_i)+h^+_\mu(x_j)-h^+_{\mu\mu}\right]$. Thus $\Re\, \U_z(x)\leq\U_1(x)$ for all $z\in G$ and all $x$ and consequently we have 
\begin{align}
\sup_{N}\E_{N,V;L}\lv \exp(\U_z)\rv\leq\sup_{N}\E_{N,V;L}\exp(\U_1)<\infty,\label{U1_bound}
\end{align}
where the last inequality follows by linearization as above. Hence the second term in the definition of $W_m$ is uniformly bounded 
on $G$. Furthermore, the uniform boundedness of the first term in \eqref{W_m} follows from $\Re\, \U_z\leq\U_1$ for all $z\in G$, 
from the convergence and hence boundedness of $W_m(1)$ and from \eqref{U1_bound}. We are now in a position to apply Vitali's 
theorem, providing the desired convergence of $W_m(-1)$ to $0$.
\end{proof}

Finally, we would like to point out a subtlety that can be explained e.g. by looking at relation \eqref{e8}. The $z$-dependence used 
above is introduced by twice replacing $\U$ by $\U_z$ on the right hand side of \eqref{e8}. Observe that this is in general not the 
same as replacing $h$ by $h_z$ on the left hand side. In fact, we never consider the ensembles $P_{N, Q}^{h_z}$ with $z \neq -1$, 
i.e. for $h_z \neq h$. The function $V$ and the associated equilibrium measure $\mu$ do not depend on $z$ either.

\section{Proofs of Results on the Largest Particle}\label{TSec4}
\begin{proof}[Proof of Theorem \ref{ethrm3}]
We have to compute the limit of the gap probability
\begin{align*}
 &P_{N,Q}^h\lbb x_j\notin(b+\frac{s}{c^*N^{2/3}},\infty),\quad j=1,\dots,N\rbb.
\end{align*}
This proof uses the same techniques and route as the proof of Theorem \ref{ethrm2}. However, in that proof the truncation threshold
$L$ depends on the order of the correlation function $k$, leading to an $\O$-term depending on $k$ in a non-obvious way, e.g.~in
the expectation of quantities like $\|f\|_D$ or $\Lip{f}$. It is therefore convenient to truncate the event before expanding in
terms of correlation functions. To this end, we recall from the last inequality in the proof of
the truncation lemma \cite[Lemma 6.3]{GoetzeVenker} that
\begin{align}
\rho_{N,Q}^{h,1}(t)\leq \exp\{{CN}-c_1N[V(t)-c_2\log(1+t^2)]\}\label{truncation_Q}
\end{align}
for some positive $C,c_1,c_2$. Hence, if $L$ is chosen large enough, we have for some $c>0$
\begin{align*}
 &P_{N,Q}^h\lbb x_j\notin[-L,L]\ \text{ for some }
j\rbb\leq 2N\int_{L}^\infty \rho_{N,Q}^{h,1}(t)dt=\O(e^{-cN}).
\end{align*}
From this we also conclude that the replacement of the normalizing constant $Z_{N,Q}^h$ by $Z_{N,Q;L}^h$
is negligible, where $Z_{N,Q;L}^h$ denotes the normalizing constant for the ensemble $P_{N,Q;L}^h$ that arises from $P_{N,Q}^h$ 
by restricting it onto $[-L,L]^N$.
We thus have
\begin{align}
 &P_{N,Q}^h\lbb x_j\notin(b+\frac{s}{c^*N^{2/3}},\infty) \text{ for all }
j\rbb=P_{N,Q;L}^h\lbb
x_j\notin(b+\frac{s}{c^*N^{2/3}},L) \text{ for all }
j\rbb+\O(e^{-cN}).\label{truncation_probability}
\end{align}
 To evaluate the latter probability, we proceed as in the proofs of Theorem \ref{ethrm2} a) and b). We will write 
explicitly only the case of negative-definite $h$, the general case follows then with the additional arguments given in the proof 
of Theorem \ref{ethrm2} a). Note that the seemingly stronger uniformity for $s\in\R$ in a) is a simple consequence of the 
continuity of $F_2$. In order to show
\begin{align*}
P_{N,Q;L}^h\lbb x_j\notin(b+\frac{s}{c^*N^{2/3}},L), j=1,...,N\rbb-F_2(s)=\O(N^{-\s})
\end{align*}
it suffices to prove the bound
\begin{align}\label{Lago5}
\E\,
\big[\E_{N,V;L}\exp\{\sum_{j=1}^Nf(x_j)\}
\lbb P_{N,V,f;L}\lbb
x_j\notin(b+\frac{s}{c^*N^{2/3}},L),j=1,...,N\rbb-F_2(s)\rbb\big]=\O(N^{-\s})
\end{align}
(cf.~\eqref{e11}).  In order to see this,  integrate \eqref{Lago1} over $[-L,b+\frac{s}{c^*N^{2/3}}]^N$ and use \eqref{Lago4} 
together with the lower bound of \eqref{Lago3}. As in the proof of Theorem \ref{ethrm2} b) we choose $\e$ and $\k$ depending on 
$\s$. Again we can neglect the contribution of those $f$ to the expectation in \eqref{Lago5} with $\|f\|_D>N^\k$.
We are therefore left to prove
\begin{align}
P_{N,V,f;L}\lbb
x_j\notin(b+\frac{s}{c^*N^{2/3}},L),j=1,...,N\rbb-F_2(s) =\O(N^{-\s})\label{boils_down}
\end{align}
uniformly for all $f \in X_D$ with $\|f\| \leq N^{\k}$ and all $s\in[q,(L-b)c^*N^{2/3}]$.
 To this end, we will represent both probabilities  as Fredholm determinants. It is well-known (see e.g. \cite{TracyWidom} for a 
nice 
derivation)
that 
\begin{align*}
 &P_{N,V,f;L}\lbb
x_j\notin(b+\frac{s}{c^*N^{2/3}},L),j=1,...,N)=\det(I-\hat{\mathcal{K}}_{N,V,f;L})=:\Delta(\hat{\mathcal{K}}_{N,V,f;L}),
\end{align*}
where $\hat{\mathcal{K}}_{N,V,f;L}$ denotes the integral operator on $L^2((s,(L-b)c^*N^{2/3}))$ with kernel
\begin{align*}
 \hat{K}_{N,V,f;L}(t_1,t_2) := \frac{1}{N^{2/3}c^*}K_{N,V,f;L}(b+\frac{t_1}{c^*N^{2/3}},b+\frac{t_2}{c^*N^{2/3}}).
\end{align*}
Observe that $\hat{K}$ agrees with the definition in Proposition \ref{keyproposition} since $b=b_V$ and $c^* = \g_V$.
For 
$F_2$ we have $F_2(s)=\Delta(\mathcal{K}_\Ai)$ with $\mathcal{K}_\Ai$ (cf. \eqref{F_2_rep}) being the integral operator w.r.t.~the 
Airy kernel on 
$L^2((s,\infty))$. For comparison, it will turn out convenient to consider Fredholm determinants of integral operators 
that are defined on the same $L^2$-space. To this end, define $\mathcal{K}_{\Ai;L}$ as the integral operator on 
$L^2((s,(L-b)c^*N^{2/3}))$ w.r.t.~the Airy kernel. Let us show that replacing $F_2(s)$ by $\Delta(\mathcal{K}_{\Ai;L})$ only 
results in a negligible error. Let $\zeta$ denote the 
determinantal point process on $\R$ determined by the 
Airy kernel (see e.g. \cite[Proposition 4.2.30]{AGZ}). Then both $F_2(s)$ and $\Delta(\mathcal{K}_{\Ai;L})$ can be represented as 
gap probabilities w.r.t.~$\zeta$. Let $P$ be the probability measure underlying $\zeta$.  From the asymptotics 
\eqref{TW_asymptotics} of $F_2$
we conclude
\begin{align*}
 &\lv F_2(s)-\Delta(\mathcal{K}_{\Ai;L})\rv=\lv P(\zeta\cap(s,\infty)=\emptyset)-P(\zeta\cap(s,(L-b)c^*N^{2/3})=\emptyset)\rv\\
&\leq P(\zeta\cap[(L-b)c^*N^{2/3},\infty)\not=\emptyset )=1-F_2((L-b)c^*N^{2/3})=\O(e^{-cN})
\end{align*}
for some $c>0$.
Thus it suffices to estimate the difference $\Delta(\hat{\mathcal{K}}_{N,V,f;L})-\Delta(\mathcal{K}_{\Ai;L})$.
Here we use again the classic series representation of Fredholm determinants \cite[Def. 3.4.3]{AGZ} and the inequality 
\cite[Lemma 
3.4.5]{AGZ} on
the difference between the Fredholm determinants of two integral operators $\mathcal{S,Q}$ with bounded kernels $S,Q$ on
an interval $(c,d)$. Let $\nu$ be a finite measure on $(c,d)$ with total mass $\|\nu\|_1$. Let $\Delta(\mathcal
S,\nu):=\det(I-\mathcal S)$ denote the Fredholm determinant of $\mathcal S$ on $L^2((c,d),\nu)$. Then
\begin{align}
 \lv \Delta(\mathcal S, \nu)-\Delta(\mathcal Q, \nu)\rv\leq
\left(\sum_{k=1}^\infty\frac{k^{1+k/2}\|\nu\|_1^k\max(\|S\|_\infty,\|Q\|_\infty)^{k-1}}{k!}\right)\|S-Q\|_\infty,
\label{Fredholm_inequ}
\end{align}
where $\|\cdot\|_\infty$ denotes the sup-norm on $(c,d)^2$. A natural choice would be to use the Lebesgue measure for $\nu$ on 
the 
intervals $(s,(L-b)c^*N^{2/3})$. However, $\|\nu\|_1$ would then be of order $N^{2/3}$ and consequently the series in 
\eqref{Fredholm_inequ} could not be bounded uniformly in $N$. For our application one can circumvent this issue by transfering 
the fast decay of the kernels onto the measure $\nu$. 
In order to do so
we work with $L^2((s,(L-b)c^*N^{2/3}),\nu)$, $d\nu(t):=e^{-2t}dt$ and set
\begin{align*}
 S_N(t_1,t_2):= \hat{K}_{N,V,f;L}(t_1, t_2)e^{t_1+t_2} 
,\qquad
Q(t_1,t_2):=K_{\Ai}(t_1,t_2)e^{t_1+t_2}.
\end{align*}
Using the above mentioned representation of the Fredholm determinant one immediately obtains
$\Delta({\hat{\mathcal{K}}_{N,V,f;L}})=\Delta({\hat{\mathcal{K}}_{N,V,f;L}},dt)=\Delta(\mathcal{S}_N,\nu)$ and
$\Delta(\mathcal{K}_{\Ai;L}){=\Delta(\mathcal{K}_{\Ai;L},dt)}=\Delta(\mathcal{Q},\nu)$. Observe that $\|\nu\|_1$ is now bounded uniformly in 
$N$. {The same holds for 
$\|Q\|_\infty$  which follows from the asymptotic behavior} of the Airy kernel as presented e.g.~in \cite[(4.23)]{KSSV}). The uniform 
boundedness of   $\|S_N\|_\infty$ can be derived 
as follows. Note first from the determinantal  formula \eqref{determinantal_relations} for $k=2$, from the positivity of the 
$2$-point correlation function and from the symmetry of the kernel that
\begin{align}
 S_N(t_1,t_2)^2 \leq S_N(t_1,t_1)S_N(t_2,t_2). \label{Testimate1}
\end{align}
The boundedness of  $S_N(t,t)$, $t \in (s,(L-b)c^*N^{2/3})$ on the diagonal can in turn be derived from the first two statements of 
Proposition \ref{keyproposition} and from Proposition \ref{prop_LD} a), where the asymptotic behavior of the Airy kernel and a 
lower bound on 
$\eta_V(x)$ of the form $c' (x-1)^{3/2}$ are being used in addition. We have 
now shown  that the series on the right hand side of \eqref{Fredholm_inequ} is uniformly bounded in $N$ if we choose $\mathcal S_N$
for $\mathcal S$. Establishing \eqref{boils_down} uniformly for $s\in[q,(L-b)c^*N^{2/3}]$ thus reduces to proving $\|S_N-Q\|_\infty = \O(N^{-\s})$. To this end recall $\e=2(\frac{1}{3} - \s - \k) > 0$. We distinguish two cases. For $t_1$, $t_2 \in [q, N^{\e}]$ Proposition \ref{keyproposition} 
yields the desired bound. In the case 
that one of the variables is larger than $N^{\e}$, the crude estimate $|S_N-Q| \leq |S_N| + |Q|$ together with 
\eqref{Testimate1} and Proposition \ref{prop_LD} a) provides a bound of the form $\O(e^{-cN^{3\e /2}})$ that suffices easily.
\end{proof}

We continue with the proof of our result for the regimes of moderate and large deviations for the upper tail of the distribution of 
the largest 
particle, which combines the analysis devised in \cite{DissSchueler,EKS} and the asymptotic results of \cite{KSSV} with the 
procedure of the previous proof.

\begin{proof}[Proof of Theorem \ref{thrm_MD}]
We start by observing that it suffices to prove the statements of Theorem \ref{thrm_MD} for $N \geq N_0$ and in the cases of 
statements a), c)  for $t \in (b + C N^{-2/3}, T)$ for arbitrarily large but fixed constants $N_0$ and $C$. Let us first turn our 
attention to statement c). We begin with the truncation of the ensemble to $[-L,L]$ as in \eqref{truncation_probability}. It 
follows from  \eqref{truncation_Q} that the constant $c>0$ in the error $\O(e^{-cN})$ of that truncation can be chosen arbitrarily 
large by increasing $L$. Thus, we may choose $L > T+1$ so large (depending on the choice of $T$ in the statement of the 
theorem) that we have 
for all $t \in (b + N^{-2/3}, T)$:
\begin{align}\label{p6b5}
\frac{P_{N,Q}^h\lbb
 x_{\max}>t\rbb}{\FNV (t)} = \frac{\E_{N,V;L}\left[\dopp{1}_A(x)e^{\U(x)}\right]}{\E_{N,V;L}e^{\U(x)}\FNV(t)} 
 \left[ 1 +\O \lb e^{-cN} \rb
 \right]
\end{align}
for some new positive constant $c$ and with $A:=\{x\in\R^N\,:\,x_{\max}>t\}$.

Now we apply the stochastic linearization to rewrite the numerator and obtain (use in particular the relation for the partition 
functions provided by \eqref{Lago4})
\begin{align}\label{p6b10}
 \E_{N,V;L}\left[\dopp{1}_A(x)e^{\U(x)}\right]=\E\left[ 
\E_{N,V;L}e^{\sum_{j=1}^Nf(x_j)}P_{N,V,f;L}(x_{\max}>t)\right].
\end{align}
In view of \eqref{p6b5}, \eqref{p6b10} we first derive upper and lower estimates on $P_{N,V,f;L}(x_{\max}>t) / \FNV(t)$.
Since $P_{N,V,f;L}$ is determinantal we have
\begin{align}\label{p6b15}
P_{N,V,f;L}(x_{\max}>t) = 1 - \Delta(\mathcal{K}_{N, V, f; L}) ,
\end{align}
where $\Delta(\mathcal{K}_{N,V,f;L})$ denotes the Fredholm determinant related to the integral operator with kernel $K_{N,V,f;L}$ 
and acting on $L^2((t,L),dx)$. Next we apply Proposition \ref{prop_LD} b). The formula there holds for all $t \in (b + C_0 
N^{-2/3}, T)$ (since we have ensured above that $L > T +1$) and for all $f \in X_D$ satisfying $\|f\| \equiv \|f\|_{D}\leq 
c_0N(t-b)$, where $c_0$, $C_0$ denote some positive constants. In this regime a first simple consequence is
\begin{align}\nonumber
\int_t^L \!\! K_{N,V,f;L}(y,y) dy=\O \lb
\frac{1}{N\lv t-b\rv^{3/2}} \rb,
\end{align}
where one should use from the proof of Proposition \ref{prop_LD} b) that
\begin{align}\nonumber
e^{-N\eta_V(t/b_V) + \O \lb \sqrt{t-b_V} \|f\| \rb} = e^{-N\eta_{V,f}\lb\l_{V,f}^{-1}(t)\rb} \leq 1.
\end{align}
The definition of the Fredholm determinant as in \cite[Def. 3.4.3]{AGZ}, Hadamard's inequality (see 
\eqref{THad}) and Proposition \ref{prop_LD} b) then yield the finer estimates
\begin{align}\label{p6b20}
1 - \Delta(\mathcal{K}_{N, V, f; L}) &= \lb \int_t^L \!\! K_{N,V,f;L}(y,y) dy\rb \left[ 1 + \O \lb
\frac{1}{N\lv t-b\rv^{3/2}} \rb 
\right]\\
&= \FNV(t) e^{\O \lb \sqrt{t-b_V} \|f\| \rb} 
\left[1 
+ \O \! \lb \frac{\|f\|}{N\lv t-b\rv} \rb
+\O \! \lb
\frac{1}{N\lv t-b\rv^{3/2}} \! \rb
\right] .
\nonumber
\end{align}
Recalling \eqref{p6b15} we have derived for all $t \in (b + C_0 N^{-2/3}, T)$ the upper bound
\begin{align}\label{p6b25}
\frac{P_{N,V,f;L}(x_{\max}>t)}{\FNV(t)} \leq 1 + \g(\|f\|) , \quad \text{with}
\end{align}
\begin{align}\label{p6b30}
\g(r) :=
\begin{cases}
\frac{C}{N \lv t-b\rv^{3/2}}&\text{for} \quad 0 \leq r < \frac{1}{N \lv t-b\rv^2} \\
Cr\sqrt{t-b}&\text{for} \quad \frac{1}{N \lv t-b\rv^2} \leq r < \frac{1}{\sqrt{t-b}} \\
e^{Cr\sqrt{t-b}}&\text{for} \quad \frac{1}{\sqrt{t-b}} \leq r < c_0N(t-b) \\
\frac{1}{\FNV(t)}&\text{for} \quad c_0N(t-b) \leq r 
\end{cases}
\end{align}
for some suitable constant $C$. In the case $r \geq c_0N(t-b)$ we have used that the numerator of the left hand side of 
\eqref{p6b25} is a probability and thus bounded by $1$.

We now formulate a basic integration lemma that is convenient for estimating the right hand side of \eqref{p6b10}. To this end 
we introduce a function $F$ that is related to the distribution function for $\|f\|$ induced by the Gaussian measure on $f$ and 
modified by the positive density $\E_{N,V;L}e^{\sum_{j=1}^Nf(x_j)}$:
\begin{align}\label{p6b35}
 F(\a) := \E \lb \dopp{1}_{\{\|f\|_D\geq \a\}} \E_{N,V;L} e^{\sum_{j=1}^N f(x_j)} \rb.
\end{align}
\begin{lemma}\label{Lp6b}
Let $0 \leq \a < \b \leq \infty$ and assume that $\g : [\a, \b) \to [0, \infty)$ is a continuously differentiable function. Then
\begin{align}\label{p6b40}
\E \lb \dopp{1}_{\{\a \leq \|f\|_D < \b\}} \E_{N,V;L} e^{\sum_{j=1}^N f(x_j)} (1 + \g(\|f\|)) \rb \leq F(\a)-F(\b) + \g(\a) F(\a) + \int_{\a}^{\b} \!\! F(r) \g'(r) dr.
\end{align}
\end{lemma}
To see this, observe first that the left hand side of \eqref{p6b40} is given by the Riemann-Stieltjes integral 
\begin{align}\nonumber
 -\int_{\a}^{\b} 1 + \g(r) dF(r) = - (1+\g) F \mid_{\a}^{\b}  + \int_{\a}^{\b} F(r) \g'(r) dr 
\end{align}
after integration by parts. Using $-\g(\b) F(\b) \leq 0$ completes the derivation of \eqref{p6b40}.

We apply Lemma \ref{Lp6b} on four intervals $[\a, \b) = [\a_{j-1}, \a_j)$ for $j =1$, $2$, $3$, $4$ according to the cases in \eqref{p6b30}, i.e.
$\a_0:=0$, $\a_1:=\frac{1}{N \lv t-b\rv^2}$, $\a_2:=\frac{1}{\sqrt{t-b}}$, $\a_3:=c_0N(t-b)$, and $\a_4:=\infty$. Combining 
\eqref{p6b5}, \eqref{p6b10}, \eqref{p6b25}, \eqref{p6b30} and using $\E_{N,V;L}e^{\U(x)} = F(0)$ (see \eqref{e7}) we obtain
\begin{align}\nonumber
\frac{\E_{N,V;L}\left[\dopp{1}_A(x)e^{\U(x)}\right]}{\E_{N,V;L}e^{\U(x)}\FNV(t)} \leq 1 &+ \g(0) + 
 \sum_{j=1}^{3}\g(\a_j)\frac{F(a_j) }{F(0)} 
 + C\sqrt{t-b} \int_{\a_1}^{\a_2} \frac{F(r)}{F(0)} dr  \\ &+ C\sqrt{t-b} \int_{\a_2}^{\a_3} e^{Cr\sqrt{t-b}}\frac{F(r)}{F(0)} dr.\nonumber
\end{align}
We are left to show that the last six summands on the right hand side are of the form $\O \left(1/(N \lv t-b\rv^{3/2})\right)$ or 
of the form $\O (\sqrt{t-b})$. Since $\g(0)$ and $\g(\a_1)$ are of the first form and trivially $F(\a_1)/F(0) < 1$ the first two of 
these summands are good. Since $\int_{\a_1}^{\a_2} F(r) dr \leq F(0)$ it is clear that the first integral in the sum is of the 
second form. For the remaining three terms we use that there exist positive constants  $d$, $K$ such that $F(r)/F(0) \leq 
K e^{-dr^2}$ for all $r \geq 0$ and for all $N$. To see this, use  the 
sub-Gaussianity of $\|f\|_D$ analogously to the derivation of 
\eqref{Gaussian_truncation} together with the lower bound on $F(0)=\E_{N,V;L}e^\U$ provided by \eqref{Lago3}.
Then it follows
\begin{align}\nonumber
 \g(\a_2) F(\a_2)/F(0) = \O \lb e^{-d/\lv t-b\rv} \rb = \O (\sqrt{t-b}).
\end{align}
Moreover, setting $x:=N \lv t-b\rv^{3/2} \geq C_0 \geq 1$, $y:=x^{4/3}$ and $c_1:= d c_0^2$, we have
\begin{align}\nonumber
 \g(\a_3) F(\a_3)/F(0) = \O \lb x e^{x} e^{-c_1 N^{2/3}y} \rb = \O \lb e^{(2-c_1 N^{2/3})y} \rb = \O \lb e^{-c_1 N^{2/3}/2} \rb
\end{align}
for $N$ sufficiently large (e.g. $N^{2/3} > 4/c_1$ would do) and is therefore of the first form. Finally, we consider the last integral in the sum. It can be estimated by
\begin{align}\nonumber
 \O \lb  \sqrt{t-b} \int_{\a_2}^{\infty} e^{-d r^2/2}dr \rb = \O \lb  \sqrt{t-b} /( d \a_2) \rb = \O \lb  \sqrt{t-b}  \rb.
\end{align}
This completes the analysis of the upper bound. The lower bound is proved in exactly the same way. One 
only needs to replace $1+\g$ by $1-\g$. In summary we obtain
\begin{align}\label{Testim5}
\frac{\E_{N,V;L}\left[\dopp{1}_A(x)e^{\U(x)}\right]}{\E_{N,V;L}e^{\U(x)}\FNV(t)} = 1 + \O\lb
\frac{1}{N\lv t-b\rv^{3/2}}\rb +\O\lb \sqrt{t-b} \rb
\end{align}
for all $t \in (b + C_0 N^{-2/3}, T)$. In view of the remark at the very beginning of the proof and using \eqref{p6b5} we have 
established statement c).

We turn our attention to the case of general functions $h$ that are not necessarily negative-definite. Observe first that 
\eqref{Testim5} still holds if $U$ is replaced by $U_z$ for any positive $z$ where we do not insist on uniformity of the $\O$-terms 
in $z$.
To see this one may use exactly the same proof of \eqref{Testim5} as above with $F$ being replaced $F_z$ which is defined via \eqref{p6b35} with the Gaussian process $f_z$ instead of $f$. One may again show that the tails of $F_z$ are sub-Gaussian and that $\inf_N F_z(0) > 0$. These were the crucial properties of $F$ used in the derivation of \eqref{Testim5}. Using in addition \eqref{U1_bound} we conclude
\begin{align}
 \frac{\E_{N,V;L}\left[\dopp{1}_A(x)e^{\U_z(x)}\right]}{\FNV(t)}-\E_{N,V;L}e^{\U_z} \to 0 \quad \text{for} \ \ N \to \infty \label{deviations_Vitali2}
\end{align}
for $0<z\leq1$ and for those values of $t$ that are considered in statement b) of the theorem. For later use we note that by the same reasoning, the left hand side of \eqref{deviations_Vitali2} remains bounded for $0<z\leq1$ and all $t \in (b + C_0 N^{-2/3}, T)$. In order to apply Vitali's theorem and in view of \eqref{U1_bound} it suffices to show boundedness of 
\begin{align}\label{deviations_Vitali}
\frac{\E_{N,V;L}\left[\dopp{1}_A(x)e^{\U_z(x)}\right]}{\FNV(t)}
\end{align}
e.g. on $G \equiv \{ z \in \C : \Re z < 1\}$. Using the pointwise estimate $\Re\, \U_z(x)\leq\U_1(x)$ for all $z\in G$ this follows 
from \eqref{deviations_Vitali2} for $z=1$ and from \eqref{U1_bound}. Vitali now yields convergence \eqref{deviations_Vitali2} also 
for $z=-1$, i.e. for $U_z$ being replaced by $U$. The lower bound of \eqref{Lago3} for $\l =1$ together with \eqref{p6b5} prove 
statement b).

Observe that we have already shown the upper bound of a). In fact, by \eqref{p6b5} and by the lower bound of \eqref{Lago3} for $\l 
=1$ we are only required to prove an upper bound on \eqref{deviations_Vitali} for $z=-1$. This can be achieved as above where one 
needs to use the boundedness of \eqref{deviations_Vitali2} on all of $(b + C_0 N^{-2/3}, T)$.

We establish the lower bound by contradiction. Assume that there is no such bound. Then, again by \eqref{p6b5}, for any fixed 
(large) $C > 1$ there is a sequence $(t_N)_N$ in $(b + C N^{-2/3}, T)$ 
such that 
\begin{align}
\lim_{N\to\infty}\frac{\E_{N,V;L}\left[\dopp{1}_{A_N}(x)e^{\U(x)}\right]}{\E_{N,V;L}e^{\U(x)}\FNV(t_N)}=0,\label{contradiction}
\end{align}
where $A_N:=\{x\in\R^N\,:\,x_{\max}>t_N\}$.
By monotonicity of $\U_z(x)$ for real values of $z$ we have for $z\in (-\infty,-1]$
\begin{align*}
\frac{\E_{N,V;L}\left[\dopp{1}_{A_N}(x)e^{\U(x)}\right]}{\FNV(t_N)}\geq  
\frac{\E_{N,V;L}\left[\dopp{1}_{A_N}(x)e^{\U_z(x)}\right]}{\FNV(t_N)}
\end{align*}
and thus, by the boundedness of $\E_{N,V;L}e^\U$, both sides converge to 0. By Vitali's theorem, this 
convergence extends to $z \in (0, 1)$ as well. However, we will prove that for  $z=1/2$, \eqref{deviations_Vitali} is bounded 
away 
from 0 uniformly in $N$ and $t \in (b + C N^{-2/3}, T)$ for sufficiently large values of $C$, thus providing the desired 
contradiction.

Recall the definition of $F_{1/2}$ below \eqref{Testim5} and
choose $\b$ such that $F_{1/2}(\b) \leq F_{1/2}(0)/2$.
In view of \eqref{p6b15} and \eqref{p6b20} there exists $C>0$ such that we have for all $t \in (b + C N^{-2/3}, T)$:
\begin{align}\nonumber
 \tilde{d}(T) := \inf_{\|f\| \leq \b}  P_{N,V,f;L}(x_{\max}>t) / \FNV(t) > 0.
\end{align}
Thus the right hand side of \eqref{p6b10} with $\U$ replaced by $\U_{1/2}$ and devided by $\FNV(t)$ is bounded below by 
$\tilde{d}(T)(1-F_{1/2}(\b)) \geq \tilde{d}(T) F_{1/2}(0)/2$. Since in addition $F_{1/2}(0) \geq F(0) =\E_{N,V;L}e^\U$  is  bounded 
away from $0$ we have arrived at the advertised lower bound on \eqref{deviations_Vitali} for $z=1/2$, completing the proof of 
Theorem \ref{thrm_MD}.
\end{proof}

\begin{proof}[Proof of Corollary \ref{a.s.convergence}]
 The corollary follows from an application of the Borel-Cantelli lemma. We thus have to show for any fixed $\e>0$
\begin{align*}
 \sum_{N=1}^\infty P_{N,Q}^h(\lv x_{\max}-b\rv>\e)<\infty.
\end{align*}
The necessary estimate for the probability of $x_{\max}>b+\e$ is provided by Theorem \ref{thrm_MD} a).
It 
remains to derive a bound for the lower tail of the distribution of $x_{\max}$. This can be done as follows. Given $\e>0$, let 
$g$ be 
a smooth, nonnegative Lipschitz function with $\int_{\R} g d\mu=1$ and with support $[b-\e+\d,b-\d]$, where $0<\d<\e/2$. Then 
\begin{align*}
 P_{N,Q}^h(x_{\max}<b-\e)\leq P_{N,Q}^h\lbb\sum_{j=1}^Ng(x_j)=0\rbb\leq P_{N,Q}^h\lbb \sum_{j=1}^Ng(x_j)-N \int 
gd\mu\geq N\rbb.
\end{align*}
With \eqref{relation}, \eqref{M1}, and \eqref{concentration_U}  an application of H\"older's inequality gives 
\begin{align*}
 P_{N,Q}^h\lbb\sum_{j=1}^Ng(x_j)-N\int 
gd\mu\geq N\rbb\leq C\left[P_{N,V}\lbb\sum_{j=1}^Ng(x_j)-N\int 
gd\mu \geq N\rbb\right]^{1/\l}
\end{align*}
for some $C>0$ and some $\l>1$.
An application of 
Chebyshev's inequality to the statement of Proposition \ref{Concentration}  with a good choice of $\e$ shows that this last 
probability is of order 
$e^{-cN^2}$ for 
some $c>0$, provided that $N$ is large enough. The corollary is proved.
\end{proof}

\appendix 
\section{Asymptotics for determinantal ensembles}\label{Sec2}
In this appendix we extract those results from \cite{KSSV} on the asymptotics of the determinantal ensembles 
$P_{N,V,f;L}:=P_{N,V-f/N;L}$, defined on $[-L,L]^N$ that are relevant in our context.

We begin by recalling the notion of the equilibrium measure with respect to a 
general convex
external field $\tilde{V}$. By this we mean the unique Borel probability
measure $\mu_{\tilde{V}}$ which minimizes the energy functional
\begin{align*}
 \mu\mapsto \int\int \log\lv t-s\rv^{-1}d\mu(t)d\mu(s)+\int \tilde{V}(t)d\mu(t).
\end{align*}
A general reference on equilibrium measures with external fields is \cite{SaffTotik}. It is known that under mild assumptions such
a unique minimizer exists. In our case $\tilde{V}$ will be convex and the density of $\mu_{\tilde{V}}$ can be described as follows (see e.g. 
\cite[Sections 1 and 2]{KSSV}): There exist real numbers $a_{\tilde{V}} < b_{\tilde{V}}$ that are uniquely determined by the two equations for $a$ and $b$,
\begin{align*}
\int_a^b \frac{\tilde{V}'(t)}{\sqrt{(b-t)(t-a)}} dt = 0 \ , \quad \int_a^b \frac{t \tilde{V}'(t)}{\sqrt{(b-t)(t-a)}} dt = 2 \pi .
\end{align*}
We denote by $\l_{\tilde{V}}$ the linear rescaling that maps the interval $[-1, 1]$ onto $[a_{\tilde{V}}, b_{\tilde{V}}]$,
\begin{align}
 \l_{\tilde{V}}(t):=\frac{b_{\tilde{V}}-a_{\tilde{V}}}{2}t+\frac{a_{\tilde{V}}+b_{\tilde{V}}}{2} \ ,\label{rescaling}
\end{align}
and introduce a function $G_{\tilde{V}}$ on $\R$ by
\begin{align}
 G_{\tilde{V}}(t):=\frac{1}{\pi}\int_{-1}^1\int_0^1\frac{(\tilde{V}\circ\l_{\tilde{V}})''(t+u(s-t))}{\sqrt{1-s^2}}duds \ ,\label{functionG}
\end{align}
that inherits real analyticity and postivity from $\tilde{V}''$. 
The density of the equilibrium measure is then given by (see \cite[(2.1), (1.11--1.14)]{KSSV})
\begin{align*}
d\mu_{\tilde{V}}(t)=\frac{2}{(b_{\tilde{V}}-a_{\tilde{V}})^2\pi}\sqrt{(t-a_{\tilde{V}})(b_{\tilde{V}}-t)}G_{\tilde{V}}(\l_{\tilde{V}
}^{-1}(t))\dopp{1}_{[a_{\tilde{V}},b_{\tilde{V}}]}(t) dt.
\end{align*}
The linearization technique of \cite{GoetzeVenker} requires to consider ensembles
$P_{N,V-f/N; L}$ that arise from restricting the $P_{N,V-f/N}$ to $[-L,L]^N$ and renormalizing them as probability measures. Note 
that the choice of $L$ throughout the paper ensures that the support $[a_{V-f/N}, b_{V-f/N}]$ of $\mu_{V-f/N}$ is contained in the 
interior of $[-L,L]$. Moreover, the functions $f$ 
are always defined on some fixed domain $D$ in the complex plane that contains $[-L,L]$. They belong to the real 
Banach space $X_D:=\{\map{f}{D}{\C}\,:\,f$ analytic and bounded, $f(D\cap\R)\subset\R\}$, equipped with the norm
 $\|f\|:=\|f\|_{D}:=\sup_{z\in D}\lv f(z)\rv$.  By \cite[Lemma
2.4]{KSSV}, the maps $\tilde{V}\mapsto a_{\tilde{V}}$, $\tilde{V}\mapsto b_{\tilde{V}}$ that relate the external fields to the 
endpoints of the support of their equilibrium measures are $C^1$ with bounded derivatives on a
neighborhood of $V$ in $X_D$. Thus, for sufficiently large values of $N$,
\begin{align}
 a_{V,f}-a_V,b_{V,f}-b_V=\O\left(\frac{\|f\|}{N}\right),\label{endpoints}
\end{align}
where the subscript $_{V,f}$ is short for $_{V-f/N}$. We can now formulate a first result on kernels $K_{N,V,f;L}$ associated to 
determinantal ensembles of the form $P_{N,V,f;L}$ that has essentially been proved in \cite{KSSV}.
\begin{prop}[cf. {\cite[Theorem 1.8]{KSSV}}] \label{keyproposition}
Let $\map{V}{[-L,L]}{\R}$ be real-analytic with $\inf_{t\in[-L,L]}V''(t)>0$. Assume that the support
$[a_V,b_V]$ of the equilibrium measure $\mu_V$ is contained in $(-L,L)$ and recall the notations introduced in this appendix. 
Moreover, define
\begin{align*}
\g_V:=\frac{(2G_V(1))^{2/3}}{b_V-a_V} \, , \quad \hat{K}_{N,V,f; L} (s, t) := \frac{1}{N^{2/3}\g_V} K_{N,V,f; L} \left(b_V+ 
\frac{s}{N^{2/3}\g_V}, 
b_V+ \frac{t}{N^{2/3}\g_V} \right),
\end{align*}
where $s$, $t \in \R$. Then we have for fixed $0<\k<1/3$, $0 < \e < \min (2/15, 2/3 -2 \k)$, $q \in \R$,  
\begin{align*}
\hat{K}_{N,V,f; L} (s, t) = 
\begin{cases}
K_\Ai(s,t) + \mathcal{O}\lb N^{\kappa - 1/3}\rb \, , &\textit{if } q\leq s, t \leq 2,\\
K_\Ai(s,t) \left( 1 + \O\lb N^{\kappa +\e /2 - 1/3} \rb  \right) \, , &\textit{if } 1\leq s, t \leq N^{\e},\\
K_\Ai(s,t)  + \Ai'(t) \O\lb N^{\kappa - 1/3} \rb \, ,   &\textit{if } q\leq s\leq 1; 2\leq t \leq N^{\e}
\end{cases}
\end{align*}
with the $\O$ terms being uniform in $N$, in $f\in X_D$ with $\|f\|_{D}\leq N^\kappa$ and in $s, t$ within the respective 
regions.
\end{prop}

\begin{proof}
The result in \cite[Theorem 1.8]{KSSV} provides leading order information on
$$\frac{1}{N^{2/3}\g_{V,f}}K_{N,V,f; L}\left(b_{V,f}+\frac{s}{N^{2/3}\g_{V,f}},b_{V,f}+\frac
{t} { N^ {2/3 } \g_{V,f}} \right)$$
together with error bounds that have the uniformity needed to prove Proposition \ref{keyproposition}.
We note in passing that \eqref{endpoints} ensures that, uniformly
for $f\in X_D$ with $\|f\|_{D}\leq N^\kappa$, $[a_{V,f},b_{V,f}]\subset (-L,L)$ for $N$ large enough. 

The reader should be aware of some differences in notation, e.g. the quantities $Q$, $V$, $\g_V^+$, $\operatorname{\mathbb{A}i}$ in 
\cite{KSSV} correspond to $V$, $V - f/N$, $\g_{V,f}(b_{V,f}-a_{V,f})/2$, $K_\Ai$ in the present paper.

Our first task is to replace $b_{V,f}$ by $b_V$ and $\g_{V,f}$ by $\g_V$.
Using
\begin{align*}
[(V-f/N)\circ \l_{V,f}]'' -  [V\circ \l_{V}]'' = [(\l_{V,f}')^2 V''\circ \l_{V,f} - 
(\l_{V}')^2 V''\circ \l_{V}]-(\l_{V,f}')^2 (f''/N)\circ \l_{V,f},
\end{align*}
relations \eqref{rescaling}-\eqref{endpoints}, the thrice differentiability of $V$ on $[-L, L]$, the analyticity of $f$ in $D$, and 
the positivity of $G_V$, one obtains 
\begin{align}
G_{V,f}(t)=G_V(t)(1+\O(\|f\|/N))\label{asymptoticsG}
\end{align}
uniformly for $f\in X_D$ with $\|f\|_{D}\leq N^\kappa$, $0<\kappa<1$ and also uniformly in $t\in \R$ with both $\l_V(t)$, 
$\l_{V,f}(t) \in [-L,L]$. This yields in particular $\g_{V,f}=\g_V(1+\O(\|f\|/N))$.
In order to be able to apply the results of \cite{KSSV} we relate to any $q \leq t \leq N^{\e}$ a number $\hat{t}=\hat{t}(N, V, f)$ 
by
\begin{align*}
 b_V+\frac{t}{N^{2/3}\g_V}=b_{V,f}+\frac{\hat{t}}{N^{2/3}\g_{V,f}}
\end{align*}
and in the same way we relate $\hat{s}$ to $s$. For all relevant values of $f$, $t$, and $s$ we obtain the uniform estimates
\begin{align}
 \hat{t}=t+t\O(\|f\|/N)+N^{2/3}\g_{V,f}(b_V-b_{V,f})=t+\O(N^{\k-1/3})\, , \quad \hat{s} = s +\O(N^{\k-1/3}).\label{Testimate2}
\end{align}
Next we discuss the different regions in the $(s, t)$ plane one by one.

{\bf (a)} $q \leq s, t \leq 2$: Application of Theorem 1.8. in \cite{KSSV} yields
\begin{align*}
\hat{K}_{N,V,f; L} (s, t) = \frac{\g_{V,f}}{\g_V} \lb K_\Ai(\hat{s},\hat{t}) + \mathcal{O} (N^{- 2/3}) \rb  
\end{align*}
and the claim follows from \eqref{Testimate2} and from $\g_{V, f}/\g_V = 1 + \O (N^{\k-1})$. Observe that we need a slightly 
modified version of the result in \cite{KSSV}, because one only has that $\hat{s}$, $\hat{t}$ lie in some small neighborhood of 
$[q, 
2]$. However, the proof in \cite{KSSV} shows that the statement also holds true in such an slightly enlarged neighborhood.
This remark applies to the next two cases as well.

{\bf (b)} $1 \leq s, t \leq N^{\e}$: Theorem 1.8 of \cite{KSSV} together with $\e < 2/15$ show
\begin{align*}
\hat{K}_{N,V,f; L} (s, t) =  K_\Ai(\hat{s},\hat{t}) \lb 1 + \mathcal{O} (N^{-1/3}) \rb .  
\end{align*}
The asymptotic formula for the Airy kernel (see e.g.~\cite[(4.23)]{KSSV}) implies the existence of some $C > 1$ with
\begin{align}
\frac{1}{C} \leq K_\Ai (s, t) (st)^{1/4} (\sqrt{s} + \sqrt{t}) \exp [\textstyle{\frac{2}{3}}(s^{3/2} + t^{3/2})] \leq C 
\label{KAi_asymp}
\end{align}
for all $s$, $t \geq 1$. Using e.g.~\cite[Proposition 4.2]{KSSV} together with the asymptotics of the Airy function and its 
derivative (see \cite[10.4.59, 10.4.61]{AbramowitzStegun}) one obtains a similar formula for the partial derivatives of the Airy 
kernel leading up to
\begin{align*}
|\nabla K_\Ai (s, t)| = K_\Ai (s, t) \O \lb \sqrt{s} + \sqrt{t}\rb =  K_\Ai (s, t) \O \lb N^{\e / 2}\rb.
\end{align*}
Since $\e < 2/3 -2\k$ relation \eqref{KAi_asymp} implies for all $s^*$, $t^*$ with distance $\O (N^{\k - 1/3})$ from $s$, resp. $t$
\begin{align*}
K_\Ai (s^*, t^*) = K_\Ai (s, t) \O (1)\, ,
\end{align*}
yielding the claim.

{\bf (c)} $q \leq s \leq 1; 2 \leq t \leq N^{\e}$: This time application of \cite{KSSV} gives for $\e < 2/15$
\begin{align*}
\hat{K}_{N,V,f; L} (s, t) = \left(  K_\Ai(\hat{s},\hat{t})  +  \Ai'(\hat{t}) \mathcal{O} (N^{-1/3}) \right) \lb1 +\mathcal{O} 
(N^{\k-1})\rb .  
\end{align*}
The standard techniques used to derive the asymptotics of the Airy kernel in \cite{KSSV} (see in particular Proposition 4.2 there) 
yield the following estimates for the values of $s$ and $t$ considered in the present case:
\begin{align*}
K_\Ai(s, t) = \Ai'(t) \O (t^{-1})\, , \quad |\nabla K_\Ai(s, t) | = \Ai'(t) \O (t^{-1/2})\, , \quad \Ai'(\hat{t}) = \Ai'(t) \O(1)\, 
,
\end{align*}
where we have used $\e < 2/3 -2\k$ to derive the last relation. This completes the proof.
\end{proof}

In the regime of moderate and large deviations Proposition \ref{keyproposition} does not provide the leading order description,  
which is the content of the following proposition. As it turns out, it suffices 
to consider the Christoffel-Darboux kernel on the diagonal.  In order to simplify the resulting formulas we now make 
use of the assumed evenness of $Q$ and $h$ that implies the evenness of $V$ (cf. Remark \ref{remark_thrm1} b)) and hence 
$a_V=-b_V$.

\begin{prop}[cf. {\cite[Theorem 1.5]{KSSV}}]\label{prop_LD}
 Let $V$ satisfy the assumptions of Proposition \ref{keyproposition} and assume in addition $a_V=-b_V$. In addition to the 
notations introduced in this appendix, recall the definition of the function $\eta_V$ in \eqref{functioneta}. There exist
constants $c_0 > 0$ and $C_0 \geq 1$ only
depending on $V$ such that for

\begin{enumerate}
\item
$b_V + C_0N^{-2/3} < t < L$:
\begin{align*}
K_{N,V,f;L}(t,t)=\frac{b_V e^{-N\eta_V(t/b_V) + \O \lb \sqrt{t-b_V} \|f\| \rb}}{4\pi (t^2-b_V^2)} 
\left[1 
+ \O \! \lb \frac{\|f\|}{N\lv t-b_{V}\rv} \rb
+\O \! \lb
\frac{1}{N\lv t-b_{V}\rv^{3/2}} \! \rb
\right].
\end{align*}
\item
$b_V + C_0N^{-2/3} < t < L - 1$:
\begin{align*}
\int_t^L \!\!\!\! K_{N,V,f;L}(y,y) dy=\frac{b_V^2 e^{-N\eta_V(t/b_V) + \O \lb \sqrt{t-b_V} \|f\| \rb}}{4\pi N (t^2-b_V^2) \eta_V'(t/b_V)} 
\left[1 
+ \O \! \lb \frac{\|f\|}{N\lv t-b_{V}\rv} \rb
+\O \! \lb
\frac{1}{N\lv t-b_{V}\rv^{3/2}} \! \rb
\right].
\end{align*}
\end{enumerate}
All $\O$ terms appearing in  statements a) or b) are uniform in $N$, in $t\in(b_V + C_0N^{-2/3}, L)$ resp. in $t\in(b_V + 
C_0N^{-2/3}, L-1)$, and in $f\in X_D$ with $\|f\| \equiv \|f\|_{D}\leq c_0N(t-b_V)$.

\end{prop}

\begin{proof}
From the arguments given in the beginning of the proof of Proposition \ref{keyproposition} it follows that we can apply the result 
\cite[Theorem 1.5 (ii)]{KSSV} to the ensembles $P_{N,V-f/N; L}$ provided that $\|f\|/N$ is sufficiently small which we 
can always achieve by choosing the constant $c_0$ in the statement of Proposition \ref{prop_LD} small enough. Hence there exists a 
positive number $\tilde{C}$ such that for all $x>1+\tilde{C}N^{-2/3}$ with $\l_{V,f}(x) \leq L$
\begin{align*}
 K_{N,V,f; L}(\l_{V,f}(x),\l_{V,f}(x))=\frac{1}{2\pi (b_{V,f}-a_{V,f})}e^{-N\eta_{V,f}(x)}\frac{1}{x^2-1}\left[1+\O\lb
\frac{1}{N\lv x - 1\rv^{3/2}}\rb\right]
\end{align*}
holds with the required uniformity.
Setting
$x=\l_{V,f}^{-1}(t)$
we immediately arrive at
\begin{align}
  K_{N,V,f; L}(t,t)=\frac{1}{2\pi} e^{-N\eta_{V,f}\lb\l_{V,f}^{-1}(t)\rb} \frac{b_{V,f}-a_{V,f}}{4(t-b_{V,f})(t-a_{V,f})}
\left[1
+\O\lb
\frac{1}{N\lv t-b_{V,f}\rv^{3/2}}\rb\right].\label{asymptoticsK}
\end{align}
Using again the smallness of $\|f\|/[N(t-b_V)]$, we learn from \eqref{endpoints} that for real exponents $\alpha$ one has
\begin{align}
(t-b_{V, f})^{\alpha} = (t-b_V)^{\alpha}\left[1 + \O_\alpha \lb \frac{\|f\|}{N\lv t-b_{V}\rv} \rb \right]\label{comparisontb}
\end{align}
and consequently
\begin{align}
\textstyle
 \frac{b_{V,f}-a_{V,f}}{4(t-b_{V,f})(t-a_{V,f})}
\left[1
+\O\lb
\frac{1}{N\lv t-b_{V,f}\rv^{3/2}}\rb\right] = \frac{b_V}{2(t^2-b_V^2)} \left[1 + \O \lb \frac{\|f\|}{N\lv t-b_{V}\rv} \rb
+\O\lb
\frac{1}{N\lv t-b_{V}\rv^{3/2}}\rb\right].\label{Taux1}
\end{align}
We now turn to the exponential term. Applying \eqref{asymptoticsG} to the definition of $\eta_V$ in \eqref{functioneta}, we see 
$\eta_{V,f}=\eta_V(1+\O(\|f\|/N))$ uniformly on the domain of interest. Thus
\begin{align*}
\eta_{V,f}\lb\l_{V,f}^{-1}(t)\rb -\eta_{V}\lb\l_{V,f}^{-1}(t)\rb = \eta_{V}\lb\l_{V,f}^{-1}(t)\rb \O (\|f\|/N) .
\end{align*}
Moreover, it is straightforward to see that $\l_V^{-1}(t) = t/b_V$, $\l_V^{-1}(t) -\l_{V,f}^{-1}(t) =  \O (\|f\|/N)$, $\eta_V (x) = 
\O(\lv x-1\rv^{3/2})$, and $\eta_V' (x) = \O(\lv x-1\rv^{1/2})$. Using in addition \eqref{comparisontb} we obtain
\begin{align*}
\eta_{V,f}\lb\l_{V,f}^{-1}(t)\rb -\eta_{V}\lb\l_{V,f}^{-1}(t)\rb =& \O \lb \lv t-b_{V}\rv^{3/2} \|f\|/N\rb, \\
\eta_{V}\lb\l_{V,f}^{-1}(t)\rb - \eta_{V}\lb t/b_V\rb =& \O \lb \lv t-b_{V}\rv^{1/2} \|f\|/N\rb,
\end{align*}
that leads to 
\begin{align}
e^{-N\eta_{V,f}\lb\l_{V,f}^{-1}(t)\rb} = e^{-N\eta_V(t/b_V) + \O \lb \sqrt{t-b_V} \|f\| \rb}.
\label{Taux2}
\end{align}
Combining  \eqref{asymptoticsK} with \eqref{Taux1} and \eqref{Taux2} completes the proof of statement a).

The essential relation for establishing the second claim is 
\begin{align}
\int_t^L \frac{e^{-N\eta(y)}}{(y-b)(y-a)} dy = \frac{e^{-N\eta(t)}}{N(t-b)(t-a)\eta'(t)} \left[1
+\O\lb
\frac{1}{N\lv t-b\rv^{3/2}}\rb\right]
\label{Taux3}
\end{align}
with an $\O$ term that is uniform in $N$, $t$, and $f$ in their respective domains and where $a = a_{V,f}$, $b = b_{V,f}$ and 
$\eta = \eta_{V,f} \circ \l_{V,f}^{-1}$. We first sketch the derivation of \eqref{Taux3} following \cite[Lemma 4.8]{DissSchueler} 
(cf. \cite{EKS}). Since $\eta$ is strictly increasing the substitution $u(y) := \eta(y)-\eta(t)$ is invertible and we obtain
\begin{align}
\int_t^L \!\! \frac{e^{-N\eta(y)}}{(y-b)(y-a)} dy = e^{-N\eta(t)} \int_0^{\eta(L) - \eta(t)} \!\!e^{-N u}k(u)du, \ \ \
k(u) := \frac{1}{(y(u)-b)(y(u)-a)\eta'(y(u))}.\nonumber
\end{align}
Observe that $k(u) = k(0) + u k'(\tilde{u})$ for some $\tilde{u} \in (0, u)$. In addition, straightforward estimates yield
\begin{align}
\frac{k'(\tilde{u})}{k(\tilde{u})} = \O \lb
\frac{1}{\lv y(\tilde{u})-b\rv^{3/2}}\rb \quad \text{and by monotonicity of $k$ and $y$:}\quad k'(\tilde{u}) = k(0) \O \lb
\frac{1}{\lv t-b\rv^{3/2}}\rb.\nonumber
\end{align}
We arrive at
\begin{align}
\int_t^L \!\! \frac{e^{-N\eta(y)}}{(y-b)(y-a)} dy = e^{-N\eta(t)} k(0) \left[ \frac{1}{N} \lb 1 + \O (e^{-dN})\rb + \frac{1}{N^2} \O \lb
\frac{1}{\lv t-b\rv^{3/2}}\rb
\right]\nonumber
\end{align}
with $d = \eta(L) - \eta(L-1) > 0$. Evaluating $k(0)$ provides \eqref{Taux3}.

Note that \eqref{Taux3} implies in particular that $\int_t^L K_{N,V,f;L}(y,y) dy$ has exactly the same representation as 
$K_{N,V,f;L}(t,t)$ in \eqref{asymptoticsK} except for additional factors $N$ and $(\eta_{V,f} \circ \l_{V,f}^{-1})'(t)$ in the denominator.
Thus the arguments in the proof of statement a) below \eqref{asymptoticsK} together with 
\begin{align}
(\eta_{V,f} \circ \l_{V,f}^{-1})'(t) &= \left( \frac{2}{b_{V,f} - a_{V,f}}\right)^2 \sqrt{(t-b_{V,f})(t-a_{V,f})} G_{V,f}(\l_{V,f}^{-1}(t)) \\
&=\frac{1}{b_V} \eta_V'(t/b_V)\left[1 
+ \O \! \lb \frac{\|f\|}{N\lv t-b_{V}\rv} \rb \right]
\nonumber
\end{align}
prove statement b).
\end{proof}

\bibliographystyle{alpha}
\bibliography{bibliography}

\end{document}